\documentclass[a4paper,cleveref, autoref, thm-restate]{lipics-v2021}
\graphicspath{ {./images/} }

\usepackage{amstext}
\usepackage{amsmath}
\usepackage{amssymb}
\usepackage{bm}
\usepackage{color}
\usepackage{colordvi}
\usepackage{comment}
\usepackage{dutchcal}
\usepackage{enumitem}
\usepackage{float}
\usepackage{graphicx}
\usepackage{graphics}
\usepackage{standalone}
\usepackage{mathtools}
\usepackage{soul}
\usepackage{svg}
\usepackage{tikz}
\usepackage{xcolor}
\usepackage{xspace}
\usepackage[ruled,vlined,noend]{algorithm2e}
\usepackage[textsize=tiny,backgroundcolor=white,linecolor=black] {todonotes} % option disable

        {\hspace*{\fill}$\Box$\par\vspace{4mm}}

\newcommand{\X}{\mathbf{X}}
\newcommand{\Y}{\mathbf{Y}}

\newcommand{\xy}{z}
\newcommand{\fset}{{\mathcal F}}
\newcommand{\sset}{{\mathcal S}}
\newcommand{\dset}{{\mathcal D}}

\newcommand{\rset}{{\mathcal R}}
\newcommand{\cset}{{\mathcal C}}
\newcommand{\bset}{{\mathcal B}}
\newcommand{\pset}{{\mathcal P}}

\newcommand{\R}{\ensuremath{\mathbb R}}

\newcommand{\N}{\ensuremath{\mathbb N}}

\newcommand{\Ver}{\Y}
\newcommand{\Hor}{\X}
\newcommand{\ver}{\nu}
\newcommand{\hor}{\sigma}
\newcommand{\horver}{\xy}

\newcommand{\Se}{\sset}
\newcommand{\Bo}{\bset}

\DeclareMathOperator{\sr}{\sf{ SR}}
\DeclareMathOperator{\gig}{\sf{ GIG}}
\DeclareMathOperator{\ugig}{\sf{ UGIG}}
\DeclareMathOperator{\chain}{\sf{ CHAIN}}
\DeclareMathOperator{\conv}{\sf{ CONV}}
\DeclareMathOperator{\prig}{\sf{ PRIG}}

\DeclareMathOperator{\dir}{DIR}
\DeclareMathOperator{\posucc}{\sf{ succ}}

\newcommand{\dist}{\mathbin{\dot{\cup}}}
\newcommand{\Sum}{\displaystyle\sum}

\newcommand\rowincludegraphics[2][]{\raisebox{-0.45\height}{\includegraphics[#1]{#2}}}

\title{On Geometric Bipartite Graphs with Asymptotically Smallest Zarankiewicz Numbers }%\thanks{This work was supported by the European Research Council (ERC) under the European Union’s Horizon 2020 research and innovation programme (grant agreement No 759557). }} %TODO Please add

\titlerunning{On Geometric Bipartite Graphs with Asymptotically Smallest Zarankiewicz Numbers} 

\author{Parinya Chalermsook}{University of Sheffield, England}{chalermsook@gmail.com}{https://orcid.org/0009-0000-2833-0472}{}%{email}{orcid}{funding}.

\author{Ly Orgo}{Aalto University, Finland}{ly.orgo@aalto.fi}{https://orcid.org/0009-0008-8888-3835}{}

\author{Minoo Zarsav}{Aalto University, Finland}{minoo.zarsav@aalto.fi}{https://orcid.org/0009-0003-1768-3864}{}

\authorrunning{P. Chalermsook et al.}

\Copyright{Parinya Chalermsook and Ly Orgo and Minoo Zarsav}


\funding{This work was supported by the European Research Council (ERC) under the European Union’s Horizon 2020 research and innovation programme (grant agreement No 759557).}

\relatedversion{} 

\nolinenumbers %uncomment to disable line numbering

\EventEditors{John Q. Open and Joan R. Access}
\EventNoEds{2}
\EventLongTitle{42nd Conference on Very Important Topics (CVIT 2016)}
\EventShortTitle{CVIT 2016}
\EventAcronym{CVIT}
\EventYear{2016}
\EventDate{December 24--27, 2016}
\EventLocation{Little Whinging, United Kingdom}
\EventLogo{}
\SeriesVolume{42}
\ArticleNo{23}
\hideLIPIcs
\begin{document}

\begin{titlepage}

\maketitle

\begin{abstract}
This paper considers the \textit{Zarankiewicz problem} in bipartite graphs with low-dimensional geometric representation (i.e., low Ferrers dimension). \footnote{Ferrers dimension, along with interval dimension and order dimension, is a standard dimensional concept in graphs. Please refer to the introduction for formal definition.}  
Let $Z(n;k)$ be the maximum number of edges in a bipartite graph with $n$ nodes and is free of a $k$-by-$k$ biclique. Note that $Z(n;k) \in \Omega(nk)$ for all ``natural'' graph classes. 
Our first result reveals a separation between bipartite graphs of Ferrers dimension three and four: while we show that $Z(n;k) \leq 9n(k-1)$ for graphs of Ferrers dimension three, $Z(n;k) \in \Omega\left(n k \cdot \frac{\log n}{\log \log n}\right)$ for Ferrers dimension four graphs (Chan \& Har-Peled, 2023) (Chazelle, 1990).
To complement this, we derive a tight upper bound of $2n(k-1)$ for chordal bipartite graphs and $54n(k-1)$ for grid intersection graphs (GIG), a prominent graph class residing in four Ferrers dimensions and capturing planar bipartite graphs as well as bipartite intersection graphs of rectangles. Previously, the best-known bound for GIG was $Z(n;k) \in O(2^{O(k)} n)$, implied by the results of Fox \& Pach (2006) and Mustafa \& Pach (2016). Our results advance and offer new insights into the interplay between Ferrers dimensions and extremal combinatorics.

\end{abstract}

\end{titlepage}

\setcounter{page}{1}

\section{Introduction}

Bipartite graphs are the most natural mathematical objects used to capture interrelations between two groups of interest. Therefore, it is unsurprising that they arise routinely in mathematical and algorithmic research communities. In extremal graph theory, bipartite graphs have been studied extensively from several perspectives.

The Zarankiewicz problem is perhaps among the oldest questions in extremal graph theory. In 1951, Zarankiewicz asked for the maximum number of edges in a bipartite graph with $n$ nodes on each side that does not contain a complete bipartite subgraph $K_{k,k}$. \footnote {Here, we are interested in a symmetric form of the Zarankiewicz problem. A more general question can involve graphs with different numbers of vertices on each side, that do not contain some biclique $K_{s,t}$.} Equivalently, given an $n$-by-$n$  matrix with 0/1 entries, that does not contain a $k$-by-$k$ all-one submatrix, what is the maximum number of $1$-entries in such a matrix? 
This question was partially answered by K\"{o}vari, S\'{o}s and Tur\'{a}n: If we denote such number by $Z(n; k)$, it was proven that $Z(n;k) \in O(n^{2-1/k})$~\cite{KST}, and the best known lower bound is $\widetilde{\Omega}(n^{2-2/(k+1)})$ shown  in~\cite{Bohman_2010}. Exact values are known for $k=2$, $k=3$, and closing the gap for $k \geq 4$ is one of the central open questions in extremal combinatorics. A trivial lower bound would be $n\cdot k$, achieved by $K_{k, n}$.

This question has also attracted significant interest in the context of specific graph classes, particularly those with additional structural constraints~\cite{chan2023number,mustafa2015zarankiewicz,fox2008separator}. Chan and Har-Peled, for instance, studied a wide range of geometrically defined bipartite graphs~\cite{chan2023number}, such as incidence graphs of points and rectangles, disks, or pseudodisks.

\subsection{Our contributions}

We consider the Zarankiewicz problem for geometric bipartite graphs. For ``natural'' graph classes, the lower bound of $Z(n;k) \in \Omega(n k)$ holds trivially.  Our results show that a nearly tight upper bound can be achieved for many natural graph classes.

To formalize the discussion, consider a bipartite intersection graph $G = (U \cup V, E)$, defined by two families of objects $\fset_U$ and $\fset_V$, where each vertex $u \in U$ (resp. $v \in V$) is associated with an object $O_u \in \fset_U$ (resp. $O_v \in \fset_V$); Let $\phi: u \mapsto O_u, v \mapsto O_v$ for $u \in U$ and $v \in V$ be the bijection between vertices and their representations. 
We say that $(\fset_U, \fset_V, \phi)$ is an \textbf{intersection representation} of $G$. Whenever the mapping $\phi$ is clear from the context, we omit it. 
\Cref{table:intersection_classes} shows bipartite intersection graphs relevant to this paper. 

\begin{table}[h]
\begin{tabular}{l|l|l|l|l}
\hline 
\textbf{Graph Classes} & \textbf{Family} $\fset_U$ & \textbf{Family} $\fset_V$ & domain & \textbf{Example} \\ \hline
    $\chain$ & rightward rays & points & $\mathbb{R}$ & \rowincludegraphics[scale=0.08]{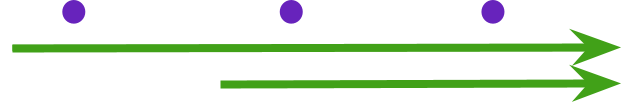} \\ \hline
    $\conv$ & intervals & points & $\mathbb{R}$ & \rowincludegraphics[scale=0.08]{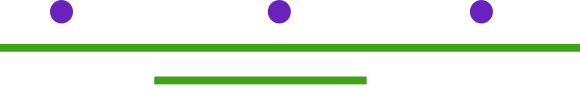} \\ \hline
    $\chain^2$ & rightward rays & upward rays & $\mathbb{R}^2$ & \rowincludegraphics[scale=0.16]{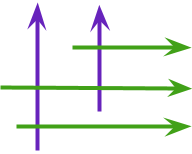} \\ \hline
    $\sr$ & horizontal segments & upward rays & $\mathbb{R}^2$ & \rowincludegraphics[scale=0.16]{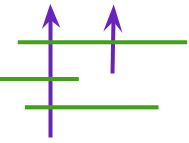} \\ \hline
    $\gig$ & horizontal segments &  vertical segments & $\mathbb{R}^2$ & \rowincludegraphics[scale=0.16]{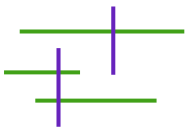} \\ \hline
    $\prig$ & axis-aligned rectangles & points & $\mathbb{R}^2$ & \rowincludegraphics[scale=0.16]{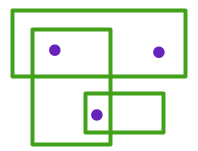} \\ \hline
\end{tabular}
\caption{Classes of bipartite intersection graphs that are relevant to this paper.}
\label{table:intersection_classes}
\vspace{-1em}    
\end{table}

For two graph classes $\cset_1 $ and $\cset_2 $, we use $\cset_1 \cdot \cset_2$ to denote the class of graphs, such that every graph $G \in \cset_1 \cdot \cset_2$ is the intersection of some graphs $ G_1 \in \cset_1$, $ G_2 \in \cset_2$, where intersection between graphs is defined as the intersection between their nodes and edges. When $\cset_1 = \cset_2$, we use the notation $\cset_1^2 = \cset_1 \cdot \cset_1$, and similarly for other powers.

The concept of Ferrers dimension is crucially based on \textbf{ chain graphs ($\chain$)}. 
A bipartite graph $G=(U\cup V, E)$ is in $\chain$ if it is an intersection graph of rays and points on the line~\cite{chaplick2014intersection}.  Ferrers dimension of a bipartite graph $G$, denoted by $\sf{fd}(G)$, is the smallest integer $d$, such that $G \in \chain^d$. It is known that every $n$-vertex graph has Ferrers dimension at most $n$.

We will use $Z_{\cset}(n,m; k, k)$ to denote the maximum number of edges in a bipartite graph $G = (U \cup V, E) \in \cset$, where $|U| = n$, $|V| = m$, such that the graph does not contain $K_{k, k}$ as a subgraph, and simply $Z_{\cset}(n; k)$, when $n = m$.

Our first main result provides a dichotomy between Ferrers dimension three and four.

\begin{restatable}{theorem}{dichotomyThm} [Dichotomy theorem] \label{thm:dichotomy}
The following dichotomy results hold for all $k\in \N$: 
\begin{itemize}
    \item %For bipartite graphs of Ferrers dimension three, 
    $Z_{\chain^3}(n;k) \leq 9n(k-1)$. 
    \item %For bipartite graphs of Ferrers dimension four, 
    $Z_{\chain^4}(n;k) \in \Omega(n k \cdot \frac{\log n}{\log \log n})$.  
\end{itemize}
\end{restatable}

By using our upper bound for Ferrers dimension three as a base case, a result by Chan and Har-Peled~\cite{chan2023number} can be improved 
%The following Corollary can be proved by induction for higher dimensions 
(proof in~\cref{sec:high dim}).
\begin{corollary}
\label{cor: high dim}
For graphs of Ferrers dimension $d \geq 3$, $k \in \N$, we have $Z_{\chain^d}(n;k) \leq O(n k \cdot \lceil\log n\rceil^{d-3})$.     
\end{corollary}

We complement the lower bound for $\chain^4$ by providing Zarankiewicz upper bounds for a prominent graph class in $\chain^4$.

\begin{restatable}{theorem}{introGIG} \label{thm:intro_GIG}
For all $k, n \in \N$, $Z_{\gig}(n;k) \leq 54n(k-1)$. % for grid intersection graphs.
\end{restatable}
This is our second main result. Grid intersection graph (GIG) is an intersection graph of horizontal and vertical segments in ${\mathbb R}^2$. This graph class contains all graphs of Ferrers dimension two and is known to be in $\chain^4$. GIG is rich, includes all planar bipartite graphs, and is equivalent to the bipartite intersection graphs of rectangles~\cite{bellantoni1993grid}. 
They have been studied from both structural and algorithmic perspectives~\cite{chaplick2017ferrers,kratochvil1989np,kratochvil1994special}. 
This result improves upon the $2^{O(k)} n$ upper bound that follows from~\cite{mustafa2015zarankiewicz,fox2010separator} (their results hold for more general intersection graphs of curves).\footnote{Note that Fox and Pach explicitly mentioned the term $2^{O(k)}$, while Mustafa and Pach's dependency on $k$ is somewhat hidden in their calculation. Looking at their calculations reveals that the constant is $2^{O(k)}$.}

Our final result applies to chordal bipartite graphs -- a bipartite graph is chordal if every cycle of length at least six contains a chord. 
Chordal bipartite graphs contain all graphs of Ferrers dimension at most two. On the other hand, it is relatively rich, having unbounded Ferrers dimension~\cite{chaplick2014intersection,chandran2011chordal}. 

\begin{restatable}{theorem}{introChordal} \label{thm:chordal-intro}
For all $k, n \in \N$, $Z_{\cset}(n;k) \leq 2n(k-1)$, where $\cset$ is a class of chordal bipartite graphs. 
\end{restatable}

We observe that $Z_{\chain}(n;k) \geq 2 n(k-1) - O(k^2)$ even for $\chain$ (see~\cref{lem: chain lower bound} in~\cref{sec: lower bound}), so our upper bound is tight up to an additive constant for fixed $k$.  
\Cref{thm:chordal-intro} improves upon the $4nk$ upper bound that holds for $\conv$~\cite{chan2023number} (since such graphs have Ferrers dimension at most two). 
%Denote by $\chain^d$ the class of bipartite graphs of Ferrers dimension at most $d$. 
Please refer to~\cref{tab:zaran} for a comprehensive list of our results and~\cref{fig:diagram} for a landscape of $Z_{\cset}(n;k)$, when $\cset$ is one of these graph chasses.

\begin{table}[h]
\begin{center}
 
\begin{tabular}{ |p{4.5cm}|p{1.6cm}|p{3.5cm}|  }
 \hline
\textbf{Graph classes}  & {\bf Our UB} & \textbf{Known UB}\\
 \hline
Chordal bipartite graphs      &$2n(k-1)$ & $3kn$ for a special case.~\cite{chan2023number} \\
Segment ray graphs (SR)     &$4n(k-1)$ & $O(kn)$~\cite{chan2023number} \\
$\chain^3$   & $9 n(k-1)$ & \\
Grid intersection graphs (GIG)  & $54 n(k-1)$ & $2^{O(k)}n$~\cite{fox2010separator,mustafa2015zarankiewicz}\\
 \hline
\end{tabular}
\caption{Zarankiewicz Bounds proved in this paper. %The proofs can be found in \cref{sec:chordal_proof}, \cref{sec:sr_proof}, \cref{sec:chain3}, and \cref{sec:gig} respectively} 
Our bounds follow from \cref{thm:chordal-intro}, \cref{cor:sr}, \cref{cor:chain3_main}, and \cref{thm:intro_GIG} respectively} 
\label{tab:zaran}
\end{center}
\vspace{-1em}
\end{table}

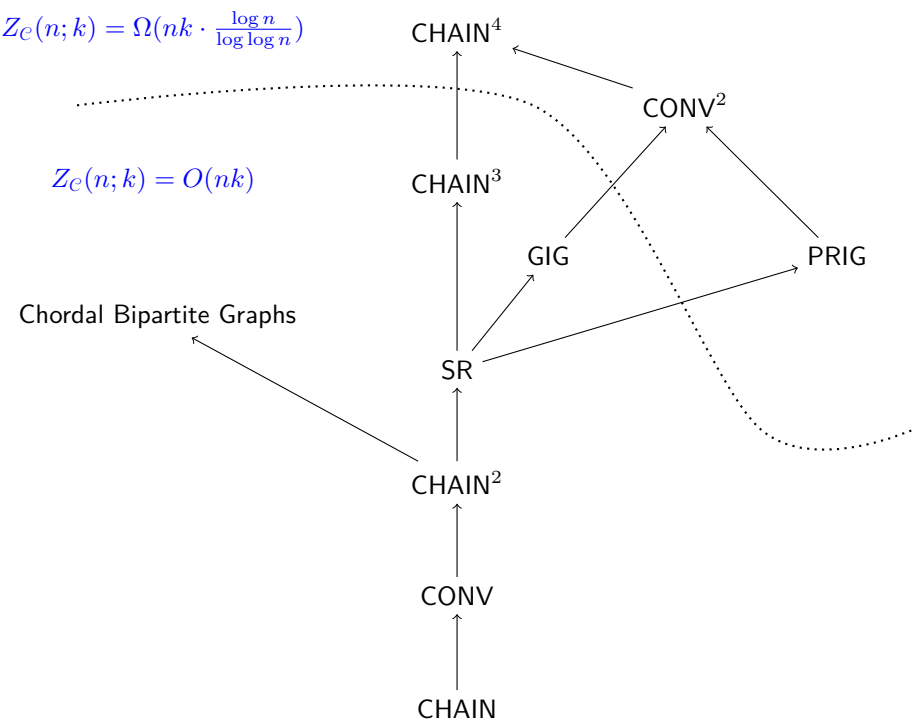
\begin{figure}[h!]
\centering
\begin{tikzpicture}

% Nodes
\node (C1) at (0,1) {$\chain$};
\node (IP) at (0,2.5) {$\conv$};
\node (C2) at (0,4) {$\chain^2$};
\node (SR) at (0,5.5) {$\sr$};
\node (C3) at (0,8) {$\chain^3$};
\node (C4) at (0,10) {$\chain^4$};
\node (CBG) at (-4,6.2) {\sf{ Chordal Bipartite Graphs}};
\node (GIG) at (1.2,7) {$\gig$};
\node (IP2) at (3,9) {$\conv^2$};
\node (RP) at (5,7) {$\prig$}; % RP positioned below IP2

%texts 

\node at (-4,10) {\textcolor{blue}{$Z_{\cset}(n;k) = \Omega (n k \cdot \frac{\log n}{\log \log n})$}};

\node at (-4,8) {\textcolor{blue}{$Z_{\cset}(n;k) = O(nk)$}};

% Edges
\draw[->] (C1) -- (IP);
\draw[->] (IP) -- (C2);
\draw[->] (C2) -- (SR);
\draw[->] (SR) -- (C3);
\draw[->] (C3) -- (C4);
\draw[->] (C2) -- (CBG);
\draw[->] (SR) -- (GIG);
\draw[->] (GIG) -- (IP2);
\draw[->] (IP2) -- (C4);
\draw[->] (SR) -- (RP);
\draw[->] (RP) -- (IP2);

% Curve
\draw[dotted, thick] plot [smooth] coordinates {(-5,9) (1,9)  (4,4.7) (6,4.7)};

%\draw[dotted, thick] plot [smooth] coordinates {(-5,3) (6,3)};

\end{tikzpicture}
\caption{The relation between graph classes considered in this paper. All inclusions denoted by arrows are known to be proper. Our results imply the separation shown by the dotted curve. We show that $\chain^3$, $\gig$, and chordal bipartite graphs satisfy $Z_{\cset}(n;k) \in O(nk)$, which implies the same for the rest of the graph classes below them. \label{fig:diagram}}
\end{figure}

We remark that all our upper bounds are algorithmic in the sense that, given a graph and its geometric representation in class $\cset$, we give an efficient algorithm that reports a $k$-by-$k$ biclique whenever the number of edges exceeds our $Z_{\cset}(n;k)$ upper bound (e.g., given a grid intersection graph $\cset = \gig$ with more than $54 (k-1)n$ edges, our algorithm efficiently finds a $k$-by-$k$ biclique).

{\bf Comparison to known results:} Prior to our results, the upper bound $Z_{\cset}(n;k) \in O(nk)$ has been known for $\conv$, segments and rays\footnote{Chan and Har-Peled use the equivalent 3-sided rectangle and point intersection graphs.}, and halfspaces and points~\cite{chan2023number}. We improve a constant for $\conv$ (since it is a special case of chordal bipartite graphs) and generalize their result on the segment-ray graphs to grid intersection graphs.
The grid intersection graph is a special case of string graphs, so the upper bound~\cite{fox2010separator,mustafa2015zarankiewicz} of $Z_{\gig}(n;k) \in O(2^{O(k)} n)$ follows from their results on string graphs.
Please refer to~\cref{fig:diagram} for the relations between these classes.

{\bf Further related works:} The graph classes in this paper have received much attention from various perspectives, including recognition algorithms, optimization, and structures. For instance, chordal bipartite graphs have been considered in~\cite{lubiw1985doubly,golovach2016enumerating,muller1996hamiltonian,uehara2005graph}. For grid intersection graphs, the recognition problem is NP-complete~\cite{kratochvil1994special,kratochvil1989np}. The class is known to capture planar bipartite graphs~\cite{hartman1991grid} and is equivalent to the intersection graph of rectangles that are bipartite~\cite{bellantoni1993grid}.  
For more detailed literature on GIG, we refer the readers to an exposition in the PhD thesis of Mustata~\cite{mustata2014subclasses}. 

Zarankiewicz problems have been studied in many other special graphs, such as semi-algebraic graphs~\cite{fox2017semi, DBLP:journals/siamdm/Do19}, graphs of bounded VC dimension~\cite{janzer2024zarankiewicz}, graphs forbidding a fixed induced subgraph~\cite{bourneuf2023polynomial}, and in geometric intersection graphs \cite{keller2023zarankiewicz, DBLP:conf/compgeom/ChanKS25, DBLP:conf/soda/ChanH23, Tomon_Zakharov_2021, chan2024zarankiewicz}. We refer to the recent survey by Smorodinsky for a more comprehensive list of related works~\cite{smorodinsky2024survey}

\section{Preliminaries}
\label{sec:prelim}

We use standard graph-theoretic notation. A bipartite graph is denoted as $G= (U \cup V, E)$ where $U$ and $V$ are partitions of the vertices. For a graph $G$, we refer to its set of vertices as $V(G)$, its set of edges as $E(G)$, and its induced subgraph on $S \subseteq V(G)$ as $G[S]$.

%The standard Zarankiewicz terminology involves four parameters $Z(m,n; s,t)$, which denotes the maximum number of edges in a bipartite graph with $m$ and $n$ nodes on the two sides, such that the graph does not contain $K_{s,t}$. 

The intersection of two graphs $G_1$ and $G_2$ is defined as $G_1 \cap G_2 = (V(G_1) \cap V(G_2), E(G_1) \cap E(G_2))$. When $G_1$ belongs in graph class $\cset_1$ and $G_2$ in class $\cset_2$, then we say that $G_1 \cap G_2$ belongs in $\cset_1 \cdot \cset_2$. For graph classes $\cset_1$ and $\cset_2$, that are closed under taking induced subgraphs, a graph $G \in \cset_1 \cdot \cset_2$ can be assumed to be the intersection of two graphs on the same vertex set, because if $G_1 \in \cset_1$, then $G_1[V(G_2)\cap V(G_1)] \in \cset_1$. Similarly for $G_2$.

%For two graphs $G$ and $G'$ on the same vertex set, we use $G \cap G'$ to denote the graph where there is an edge if and only if the edge appears in both $G$ and $G'$.

Therefore, for two graph classes $\cset_1$ and $\cset_2$ that are closed under taking induced subgraphs, the intersection of the graph classes can be defined as:
\[\cset_1 \cdot \cset_2 = \{G_1 \cap G_2: V(G_1)= V(G_2), G_1 \in \cset_1, G_2 \in \cset_2\}\] 
One can naturally write $\cset^d = \cset \cdot \cset \ldots \cset$ ($d$ times).

Bicliques are the graphs $K_{a, b}$, where $a$ and $b$ are some positive integers. 

\begin{proposition}
If the graph class $\cset$ contains every biclique, then $\cset\subseteq \cset^2 \subseteq \ldots$. 
\end{proposition} 
\begin{proof}
Here, we show that for any positive integer $i$,  $\cset^{i} \subseteq \cset^{i+1}$. Let $G = (U \cup V, E) \in \cset^i$. Then take $H$ being the complete bipartite graph on $U$ and $V$ (where the edges connect vertices in $U$ and $V$). Then  $H \in \cset$. Therefore, $G = G\cap H \in \cset^{i+1}$.      
\end{proof}

Since $\chain$ contains every biclique, this implies that $\chain \subseteq \chain^2 \subseteq \ldots$. It is known that $\chain^2$ is equivalent to the class of all two-directional orthogonal ray graphs (2DOR) \cite[Theorem 2.1]{chaplick2014intersection}.

%Here, we show that $\cset^{i} \subseteq \cset^{i+1}$. 
%Let $G \in \cset^{i+1}$, according to~\cite{chaplick2017ferrers} for graph classes that are closed under taking disjoint union, we can assume that $G = G_1 \cap G_2$ where $G_1 \in \cset$ and $G_2 \in \cset^{i}$ (with the same bipartition as $G$). If the class $\cset$ includes the biclique and assume that $G_1$ is that biclique, then $G_2 = G \in \cset^{i+1}$.

%\begin{proposition}
%For every bipartite graph $G= (A\cup B,E)$, ${\sf fd}(G) \leq |A||B| - |E(G)|$.    
%\end{proposition}
%\begin{proof}
%Since $K_{n, m}$ with one edge removed is in $\chain$, then such an integer $d$ exists for every bipartite graph.
%\end{proof}

A bipartite graph $G= (U \cup V, E)$ is \textbf{convex} over $V$ if 
there exists an ordering $V = \{v_1, \ldots v_{|V|}\}$ such that for all $u \in U$, the vertices $N(u)$ are contiguous ($\{v_i, \ldots v_{j}\}$ for some $i, j \in [|V|]$). We call the set of such graphs $\conv$. $\conv$ is known to be equivalent to a containment graph of points and intervals in $\R$. 
%We say that $(\Hor, \Ver, \phi)$ is a point-segment representation of a graph $G = (U \cup V, E)$, if $\Hor$, $\Ver$ are sets of points and intervals, respectively, in $\mathbb{R}$ , $\phi: U \mapsto \Hor, V \mapsto \Ver$, and for all $u\in U$, $v\in V$, we have $\{u, v\} \in E$ if and only if $\phi(u)$ and $\phi(v)$ intersect.
%The class of graphs with a point-segment representation is exactly $\conv$ \cite{chaplick2014intersection}.
%Analogously, we can define the convexity over $B$.
%Convex bigraphs are known to be represented as intersection bigraphs of intervals and points \cite{chaplick2014intersection}.
An interesting graph class worth mentioning is $\conv^2$ -- the intersection of two convex graphs. 

\begin{theorem} \label{thm:conv2}
    A bipartite graph $G = (U \cup V, E)$ is in $\conv^2$ if and only if $G$ is a disjoint union of $\gig$ and $\prig$ graphs.
\end{theorem}
The proof can be found in \cref{sec:conv2}.

%\begin{proposition} \label{convex}
%Let $G= (U \cup V, E)$ be a graph. Then $G$ is in $\conv^2$ if and only if each connected component of $G$ is $\gig$ or an intersection bigraph of rectangles and points.      
%\end{proposition}

Finally, a bipartite graph is a \textbf{chordal bipartite graph} if every induced cycle of length at least six contains a chord. It is known that chordal bipartite graphs can have unbounded Ferrers dimension~\cite{chaplick2014intersection,chandran2011chordal}.

\section{Warm-up}
\label{sec:warm-up}

Recall that a graph $G$ is $d$-\textbf{degenerate} if every induced subgraph has a vertex of degree at most $d$. All graph classes considered in this paper are closed under taking induced subgraphs.
%Every $d$-degenerate graph $G$ contains at most $d|V(G)|$ edges. 

\begin{observation}
\label{obs: degenrate counting}
A $d$-degenerate graph $G$ has at most $d \cdot |V(G)|$ edges.
%Let $G$ be a bipartite graph with $n$ vertices. If $G$ is $d$-degenerate, it has at most $dn$ edges.     
\end{observation}

\subsection{Chordal bipartite graphs} \label{sec:chordal_proof}

For a bipartite graph $G$, denote its biadjacency matrix by $M_G$. 
We say that a matrix $M$ \textbf{contains} submatrix $P$ if we can obtain $P$ by removing some rows and columns of $M$. A matrix $M$ is $P$-\textbf{free} if it does not contain submatrix $P$. A bipartite graph $G$ is $P$-\textbf{freeable}, if there exists an ordering of rows and columns, such that the biadjacency matrix $M_G$ is $P$-free.

The following structural result is known for chordal bipartite graphs. 

\begin{theorem}[\cite{klinz1995permuting, DBLP:journals/dm/Farber83}]
\label{thm: chordal freeable}
Every chordal bipartite graph is  $\begin{psmallmatrix}
0 & 1 \\
1 & 1 
\end{psmallmatrix}$-freeable. 
\end{theorem}

\begin{lemma} \label{chordal_degree}
For any $k \in \N$, a chordal bipartite graph that does not contain $K_{k,k}$ is $(k-1)$-degenerate.
\end{lemma}
\begin{proof} Let $H$ be a chordal bipartite graph. If for every induced subgraph of $H$, a node of degree at most $k-1$ exists, we are done. Therefore, let us assume by contradiction that $G$ is an induced subgraph of $H$, where a node of degree at most $k-1$ does not exist. Then all nodes $V(G)$ have a degree of at least $k$ in $G$.
Using~\cref{thm: chordal freeable}, we know that there is a $P$-free biadjacency matrix of $G$, where $P=\begin{psmallmatrix}
0 & 1 \\
1 & 1 
\end{psmallmatrix}$. Let $M_G$ be that biadjacency matrix. Let $G = (U \cup V, E)$, where each row $i \in [|U|]$ and column $j \in [|V|]$ correspond to the vertices $u_i \in U$ and $v_j \in V$, respectively.
Let $i \in [|U|]$ be the maximum integer, such that $M_G(i, |V|) = 1$ (i.e., this is the bottommost row in the last column such that the entry is one). 
Consider the rows $R= \{i': u_{i'} \in N_G(v_{|V|})\}\subseteq [|U|]$ and columns $C= \{j': v_{j'} \in N_G(u_i)\}\subseteq [|V|]$. 
\begin{figure}[h]
    \centering
    \includegraphics[scale=0.18]{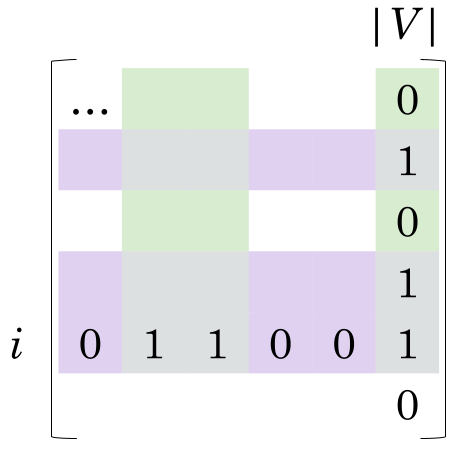}
    \caption{An example, where rows $R$ are denoted in purple and columns $C$ in green.}
    \label{fig:chordal_example}
\end{figure}
Let $M'$ be the submatrix induced on rows $R$ and columns $C$. 
According to our assumption, this submatrix $M'$ contains at least $k$ rows and $k$ columns. 
It is easy to see that $M'$ is an all-one matrix: Assume that it is not, and $M_G(a,b) = 0$ for some $a \in R$ and $b \in C$. Since $a \in R$, we have $M_G(a,|V|) = 1$ and since $b \in C$, $M_G(i,b)=1$.  
The $2$-by-$2$ submatrix induced on rows $\{a,i\}$ and columns $\{b,|V|\}$ is then equal to the forbidden pattern $P$, a contradiction. 
\end{proof}

\begin{theorem} \label{thm:chordal}
For all $n, m, k \in \N$,  $Z_{\cset}(m,n;k, k) \leq (m + n)(k-1)$, where $\cset$ is the class of chordal bipartite graphs.  
\end{theorem}
\begin{proof}
    The result follows from \cref{chordal_degree} and \cref{obs: degenrate counting}.
\end{proof}

The matching lower bound is shown in \cref{lem: chain lower bound}.

%\vspace{1em}
%\textbf{\Cref{thm:chordal-intro}}.
%$Z(n;k) \leq 2n(k-1)$ for chordal bipartite graphs.  
\introChordal*
\begin{proof}
    The result follows from \cref{thm:chordal} when $n = m$.
\end{proof}

%\begin{corollary}\label{cor:chordal}
%When $m = n$, $Z(n;k) \leq 2n(k-1)$ for chordal bipartite graphs.  
%\end{corollary}

\subsection{Segment ray graphs} \label{sec:sr_proof}

Next, we present a simple proof of the upper bound of $Z_{\sr}(m, n ; k, k)$ for the intersection graph of horizontal segments and vertical upward rays in ${\mathbb R}^2$, denoted by $\sr$.
More precisely, $(\sset, \rset, \phi)$, is a segment ray ($\sr$) representation of a bipartite graph $G= (U \cup V, E)$, if $\sset$ and $\rset$ are sets of horizontal segments and vertical upwards rays in $\R^2$ respectively, and $\phi: U \mapsto \sset, V \mapsto \rset$, such that $\{u, v\} \in E$, where $u \in U$, $v \in V$, if and only if $\phi(u)$ and $\phi(v)$ intersect.
%More precisely, given a graph $G= (A \cup B, E)$, each vertex $a \in A$ is associated with a horizontal segment $s_a$, and each vertex $b \in B$ with a vertical upward ray $r_b$. There is an edge $\{a, b\}$ if and only if $s_a$ and $r_b$ intersect.
%See Figure~\ref{fig:SR example} for an example of a segment-ray graph.

Although the upper bound $O((m + n)k)$ was shown in~\cite{chan2023number}, we improve this upper bound to $2(m + n)(k-1)$.

\begin{lemma} \label{sr_degree}
Every segment-ray graph that does not contain $K_{k,k}$ is $2(k-1)$-degenerate. 
\end{lemma}

\begin{proof}
Let $H$ be a segment-ray graph. If, for every induced subgraph of $H$, a node of degree at most $2(k-1)$ exists, we are done. Therefore, let us assume by contradiction, that $G$ is an induced subgraph of $H$, where a node of degree at most $2(k-1)$ does not exist. Then, all nodes $V(G)$ must have a degree at least $2k -1$ in $G$. 

Let $(\sset, \rset, \phi)$ be the $\sr$ representation of $G = (U \cup V, E)$, such that $\phi: U \mapsto \sset, V \mapsto \rset$. Let $x$ and $y$ be the horizontal and vertical axes, respectively.
%For a segment-ray graph $G = (U \cup V, E)$, let $\sset$ and $\rset$ be the sets of segments and rays representing $U$ and $V$, respectively.  
%Assume by contradiction that all nodes have a degree at least $2k -1$. 
Each ray $r\in\rset$ can be addressed by a pair $(r_x,r_y)$, where $r_x$ and $r_y$ are the $\bold{x}$ and $\bold{y}$ coordinates of its starting point respectively. 
Each horizontal segment $s\in\sset$ can be addressed using three variables $(s_l, s_r, s_y)$, such that $[s_l, s_r]$ and $s_y$ are the projections of $s$ to the $x$ and $y$ axis respectively. 
%showing the $\bold{x}$ coordinates of the left and right endpoints, and their $\bold{y}$ coordinate, respectively.

Let $r \in \rset$ be the ray with the largest $r_y$ (highest starting point). 
Let $S$ be the set of segments that intersect $r$. Then, according to our assumption, $|S|\geq 2k-1$. 
Let $S_{right} \subseteq S$ be the set of $k-1$ segments with the largest (rightmost) $s_r$, and let $S_{left} \subset S$ be the set of $k-1$ segments with the least (leftmost) $s_l$. 
Then there must be at least one segment $s \in S \setminus (S_{right} \cup S_{left})$ (\cref{fig:SR_degree}). 
%Notice that each segment in $S_{left}$ has the left endpoint to the left of the left endpoint of $s$. 
\begin{figure}[h]
    \centering
    \includegraphics[scale=0.18]{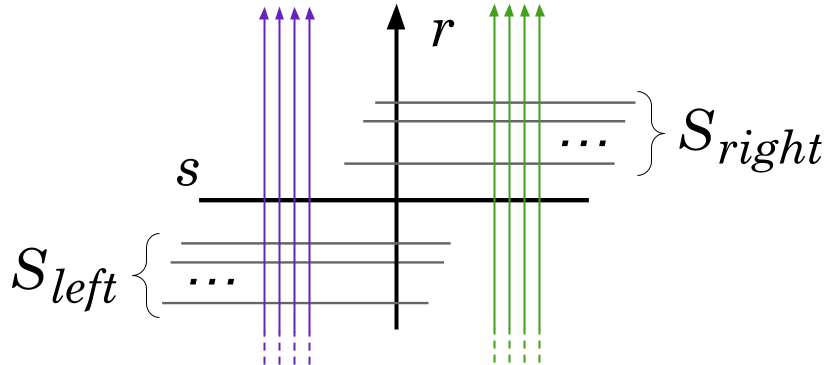}
    \caption{$R_{left}$ and $R_{right}$ are denoted by purple and green rays respectively. }
    \label{fig:SR_degree}
\end{figure} 

Let $R$ be the set of rays that intersect $s$. By our assumption, $|R| \geq 2k-1$. 
Let $R_{left}, R_{right} \subseteq R$ be the sets of rays to the left and right of $r$, respectively.  
Since $R = R_{left} \cup R_{right} \cup \{r\}$, then at least one of $R_{left}, R_{right}$ must contain at least $k-1$ rays. Let us assume it is $R_{left}$, the other case can be argued symmetrically.
Since $r$ is the ray with the highest $r_y$, then each ray in $R_{left}$ must intersect each segment in $S_{left}$ (\cref{fig:SR_degree}). This implies that the vertices mapped to rays in $R_{left} \cup \{r\}$ and segments in $S_{left} \cup \{s\}$ must induce a $k$-by-$k$ biclique. This is a contradiction. 
\end{proof}

\begin{theorem} \label{thm:sr}
For all $n, m, k \in \N$, $Z_{\sr}(m,n;k, k) \leq 2(m + n)(k-1)$ in segment ray graphs.  
\end{theorem}
\begin{proof}
    The result follows from \cref{sr_degree} and \cref{obs: degenrate counting}.
\end{proof}

\begin{corollary} \label{cor:sr}
When $m = n$, $Z_{\sr}(n;k) \leq 4n(k-1)$ in segment ray graphs.  
\end{corollary}

\section{Graphs of Ferrers Dimension Three}\label{sec:chain3}

This section presents the proof of the first part of \cref{thm:dichotomy}.

Let $x$ and $y$ be the horizontal and vertical axes, respectively.
We say that a rectangle $r$ in $\mathbb{R}^2$ is bottomless \footnote{These rectangles have also been called $3$-sided rectangles \cite{chan2023number} in the context of $\sr$ graphs.} if its projection to the $y$-axis is an interval that starts at $-\infty$.

%We say that $(\Hor, \Ver, \phi)$ is a bottomless rectangle containment representation of a graph $G = (U \cup V, E)$, if $\Hor$, $\Ver$ in $\mathbb{R}^2$ are bottomless rectangles, $\phi: U \mapsto \Hor, V \mapsto \Ver$, and for all $u\in U$, $v\in V$, we have $\{u, v\} \in E$ if and only if $\phi(u)$ is contained within $\phi(v)$. A graph $G$ is a bottomless rectangle containment bigraph, if it has a bottomless rectangle containment representation. 

We say that $(\Se, \Bo, \phi)$ is a segment bottomless rectangle containment representation of a bipartite graph $G = (U \cup V, E)$, if $\Se$ and $\Bo$ are sets of horizontal segments and bottomless rectangles in $\mathbb{R}^2$, respectively, $\phi: U \mapsto \Se, V \mapsto \Bo$, and for all $u\in U$, $v\in V$, we have $\{u, v\} \in E$ if and only if $\phi(u)$ is contained within $\phi(v)$. %Note that $(\Se, \Bo, \phi)$ is equivalent to a bottomless rectangle containment representation. Therefore, in order to more clearly distinguish the partitions, we use segments instead of bottomless rectangles for the partition that is contained.
A bipartite graph $G$ is a segment bottomless rectangle containment graph, if it has a segment bottomless rectangle containment representation. 

For $s \in \Se \cup \Bo$, let $s.x$ and $s.y$ denote the projections of $s$ to $x$- and $y$-axis respectively. For an interval $s.x = [a, b]$, we use $\min(s.x) = a$ and $\max(s.x) = b$ to denote its minimum and maximum points.%, e.g $s.x = [s.x_{min}, s.x_{max}]$.

%Our proof requires the concept of containment bigraph. 
%A \textbf{containment bigraph} $G = (A \cup B, E)$ is defined by families $\fset_A$ and $\fset_B$ where $a \in A$ is associated with $O_a \in \fset_A$ (and similarly for $b \in B$); there is an edge $\{a,b\}\in E$ if and only if $O_a \subseteq O_b$. 
%Note the difference between intersection bigraphs and containment bigraphs: In containment bigraphs, we do not care about intersections that are not containment, so for a given geometric representation, the set of edges in its containment bigraph is always a subset of the set of edges in its intersection bigraph. Again, let $\phi$ be the mapping that maps each $a$ to $O_a$ and $b$ to $O_b$. 
%In this case, we say that $(\fset_A, \fset_B, \phi)$ is a containment representation of $G$. 

\begin{lemma} \label{lem:chain3_is_brc}
A bipartite graph $G$ is of Ferrers dimension three ($G \in \chain^3$) if and only if it is a segment bottomless rectangle containment graph.  
\end{lemma}

The proof can be found in~\Cref{sec:chain3_is_bot}.

Given a bipartite graph $G \in \chain^3$, we now know it must have a containment representation $(\Se, \Bo, \phi)$, where $\Se$ and $\Bo$ are sets of horizontal segments and bottomless rectangles in $\mathbb{R}^2$, and $\phi: U \mapsto \Se, V \mapsto \Bo$. 

Let $k$ be the minimum positive integer such that $G$ does not contain $K_{k, k}$ as a subgraph.

Without loss of generality, assume that no two horizontal segments have the same $y$-coordinate. We define three (strict) partial orders (denoted by $\pset^{DL}$, $\pset^{DR}$, and $\pset^{C}$) on $\sset$ that will be crucial to our analysis (\Cref{fig:CHAIN3_segment_orderings}). 
For $s, s' \in \sset$, we say
\begin{itemize}
    \item[$\bullet$] $s'$ succeeds $s$ in the partial order $\pset^{DL}$, and denote it by $s \prec_{DL} s'$, if $\min(s'.x)\leq \min(s.x)$, $\max(s'.x) \leq \max(s.x)$, and $s'.y < s.y$, 

    \item[$\bullet$] $s'$ succeeds $s$ in the partial order $\pset^{DR}$, and denote it by $s \prec_{DR} s'$, if $\min(s.x) \leq \min(s'.x)$, and $\max(s.x) \leq \max(s'.x)$, and $s'.y < s.y$, 
    
    \item[$\bullet$] $s'$ succeeds $s$ in the partial order $\pset^{C}$, and denote it by $s \prec_{C} s'$, if $s'.x \subset s.x$.
\end{itemize}
It is easy to check that $\pset^{DL}$, $\pset^{DR}$, and $\pset^{C}$ are transitive and asymmetric.

%Consider a containment representation bigraph of horizontal segments and bottomless rectangles, $(\sset, \bset, \phi)$.
%Each segment $s\in\sset$ is addressed by $(s_l, s_r, s_y)$, showing the $\bold{x}$-coordinates of the left and right endpoints, and their $\bold{y}$-coordinate, respectively.
%Each bottomless rectangle $b\in\bset$ is addressed by $(b_l, b_r, b_y)$ showing the corresponding coordinates of its top segment.

%Let $G \in \chain^3$ be a graph not containing a $k$-by-$k$ biclique represented by $(\sset, \bset, \phi)$. 
%Without a loss of generality, assume that no two horizontal segments have the same $y$-coordinate.
%We define three (strict) partial orders (denoted by $\pset^{DL}$, $\pset^{DR}$, and $\pset^{C}$) on $\sset$ that will be crucial to our analysis (\Cref{fig:CHAIN3_segment_orderings}). 
%For $s, s' \in \sset$, we say
%\begin{itemize}
%    \item[$\bullet$] $s'$ succeeds $s$ in the partial order $\pset^{DL}$, and denote it by $s \prec_{DL} s'$, if $s'_l\leq s_l$, $s'_r \leq s_r$, and $s'_y < s_y$, 

%    \item[$\bullet$] $s'$ succeeds $s$ in the partial order $\pset^{DR}$, and denote it by $s \prec_{DR} s'$, if $s_l \leq s'_l$, and $s_r \leq s'_r$, and $s'_y < s_y$, 
    
%    \item[$\bullet$] $s'$ succeeds $s$ in the partial order $\pset^{C}$, and denote it by $s \prec_{C} s'$, if $s_l \leq s'_l $ and $s'_r\leq s_r$; in this case, the projection of $s$ on the $x$-axis contains that of $s'$. 
%\end{itemize}
%It is easy to check that $\prec_{DL}$ and $\prec_{DR}$ are transitive and asymmetric. 

\begin{figure}[h]
    \vspace{-1em} 
    \centering
    \includegraphics[scale=0.18]{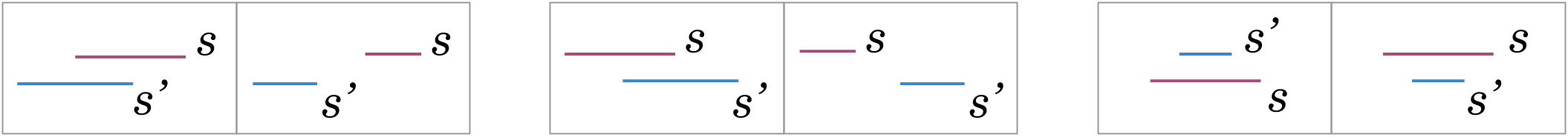}
    \caption{The scenarios that define relations in $\pset^{DL}$ (two leftmost), $\pset^{DR}$ (two in the middle) and $\pset^{C}$ (two rightmost images), where $s'$ succeeds $s$.  
    }
    \label{fig:CHAIN3_segment_orderings} 
\end{figure} 
If $s \prec_{DL} s'$ or $s' \prec_{DL} s$, we say that $s$ and $s'$ are comparable in $\pset^{DL}$; otherwise, they are incomparable. We can say the same about comparability in $\pset^{DR}$ and $\pset^C$.  
The following claim is straightforward, and a consequence of a simple case analysis (see \cref{fig:CHAIN3_segment_orderings}).  

\begin{observation} \label{lem:chain3_all_ordered}
A pair of segments $s, s' \in \Se$ is comparable in one of the partial orders $\pset^{DL},\pset^{DR},\pset^C$. 
\end{observation}

For a segment $s \in \Se$, denote by $\posucc^{DL}(s), \posucc^{DR}(s), \posucc^{C}(s) \subseteq \Se$ the sets of all successors of $s$ with respect to partial orders $\pset^{DL}$, $\pset^{DR}$ and $\pset^{C}$, respectively.

To count the number of edges in a bipartite graph $G$ with a representation $(\Se, \Bo, \phi)$, we classify the edges into \textbf{bulky} and \textbf{thin} edges. 
For an edge $\{u, v\}\in E(G)$, we have containment $\phi(u) \subseteq \phi(v)$, where $\phi(u) \in \Se$, $\phi(v) \in \Bo$. We say that the edge $\{u, v\}$ is \textbf{DL-bulky} if the bottomless rectangle $\phi(v)$ contains $k-1$ segments from $\posucc^{DL}(\phi(u))$ (See \cref{fig:DL_bulky}).
We can analogously define \textbf{DR-bulky} and \textbf{C-bulky} edges. All edges that are not bulky are called thin.
An edge is bulky if it is DL-bulky, DR-bulky, or C-bulky.  

\begin{figure}[h]
    \vspace{-1em} 
    \centering
    \includegraphics[scale=0.18]{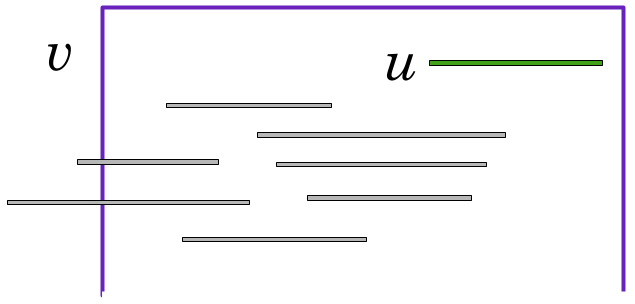}
    \caption{The edge $\{u, v\}$ is DL-bulky, if $\phi(v)$ (purple) contains $k-1$ segments succeeding $\phi(u)$ (green) in $\pset^{DL}$.}
    \label{fig:DL_bulky} 
\end{figure} 

If $s \prec_{DL} s'$ or $s' \prec_{DL} s$, we say that $s$ and $s'$ are comparable in $\pset^{DL}$; otherwise, they are incomparable. We can say the same about comparability in $\pset^{DR}$ and $\pset^C$.  
The following claim is straightforward, and a consequence of a simple case analysis (see \cref{fig:CHAIN3_segment_orderings}).  

\begin{lemma} \label{lem:chain3_bulky}
For each $u \in U$, the number of bulky edges incident to $u$ of each type is at most $(k-1)$.  Therefore, at most $3(k-1)$ bulky edges are incident to each $u$.    
\end{lemma}
\begin{proof}
We present the proof for DL-bulky edges. The proofs for the other two cases are similar.
Let $B \subseteq N_G(u)$ be the nodes $v \in V$, such that $\{u, v\}$ is DL-bulky. 
%Let $B \subseteq N_G(s)$ be the collection of rectangles $b$ containing $s$ such that $\{s,b\}$ is DL-bulky. 
Let $S' \subseteq \posucc^{DL}(\phi(u))$ be the $k-1$ segments $s' \in S'$ with the largest $\min(s'.x)$. Let $U'$ be the set of nodes represented by $S'$.

We will show that each $v \in B$ has $\{u\} \cup U'$ in its neighborhood, and therefore $|B| \leq k-1$.
Assume by contraposition that $|B| \geq k$. Let $v\in B$ be an arbitrary node.
%For $b\in B$, since $\{s,b\}$ is DL-bulky, $b$ contains at least $k-1$ DL-successor  segments of $s$ (i.e., segments in $\posucc^{DL}(s)$). 
%Consider a $(k-1)$-size set $S' \subseteq \posucc^{DL}(s)$ with largest $s'_l$ (with the rightmost left endpoints) for $s'\in S'$.
For any $s'\in S'$, $\phi(u) \prec_{DL} s'$ and $\phi(u) \subseteq \phi(v)$ imply the following:
\begin{itemize}
    \item $s'.y \in \phi(v).y$. Since $\phi(u) \prec_{DL} s'$, we have $s'.y\leq \phi(u).y$, and since $\phi(v)$ contains $\phi(u)$, we have $\phi(u).y \in \phi(v).y$. Since $\phi(v).y$ is an interval that starts at $- \infty$, then $s'.y \in \phi(v).y$.
    \item $\max(s'.x) \leq \max(\phi(v).x)$. Similarly to the previous case, the implications give us the inequalities in $\max(s'.x) \leq \max(\phi(u).x) \leq \max(\phi(v).x)$
\end{itemize}

Let $s' \in S'$ be an arbitrary segment.
Since $\phi(v).y$ contains $s'$ on the $y$-axis,  $\max(\phi(v).x)$ is larger than any point in $s'$ on the $x$-axis, then to show that $\phi(v)$ contains $s'$, it only remains to show that $\min(\phi(v).x) \leq \min(s'.x)$. 
Since $S'$ was chosen to contain the segments in $\posucc^{DL}(\phi(u))$ with the largest $\min(s'.x)$, then $\min(\phi(v).x) \leq \min(s'.x)$ holds for $s' \in S'$, if it holds for any $k-1$ segments in $\posucc^{DL}(\phi(u))$. Since $\{u, v\}$ is DL-bulky, $\phi(v)$ contains at least $k-1$ segments in $\posucc^{DL}(\phi(u))$. This implies that $\phi(v)$ contains $S'$, and we have $U' \in N_G(v)$. Since this is true for all $v \in B$, then $B$ and $U'\cup\{u\}$ induce a $K_{k,k}$, which is a contradiction. 

The proof for DR-bulky edges is symmetric and the proof for C-bulky edges can easily be seen by selecting $S' \subseteq \posucc^{C}(\phi(u))$ to be the $k-1$ segments $s' \in S'$ with the smallest $s'.y$. Since every $s' \in S'$ is already contained in $\phi(u)$ on the $x$-axis, then the $y$ coordinate is the only factor determining which elements of $\posucc^{C}(\phi(u))$ are contained within a bottomless rectangle that contains $\phi(u)$. 

%(i.e., segments in $\posucc^{DL}(\phi(u))$). Since $\phi(v).y$ contains any segment in $S'$ on the $y$-axis,  $\max(\phi(v).x)$ is larger than any point in a segment in $S'$, and $S'$ is chosen to contain the DL-successor segments of $\phi(u)$ with the smallest $\min(s'.x)$, then $\phi(v)$ must contain $S'$.
%Since $s'$ succeeds $s$, we have $s'_y\leq s_y$, and since $b$ contains $s$, we have $s_y\leq b_y$, which means $s'_y\leq b_y$.
%We know that $b$ must contain a collection of $k-1$ DL-successor segments of $s$; therefore, it contains those with the rightmost left endpoints.
%$B$ and $U'\cup\{u\}$ induce a $K_{k,k}$, which is a contradiction. 
%This implies that $B$ and $S'\cup\{s\}$ form a $k$-by-$k$ biclique, which is a contradiction. 
\end{proof}

\begin{lemma} \label{lem:chain3_thin}
For each $v \in V$, there are at most $6(k-1)$ thin edges incident to $v$. 
\end{lemma}
\begin{proof}
For a partial order $\pset$, we use $E(\pset)$ to denote the number of comparable pairs in $\pset$.

Let $T\subset N_G(v)$ be the set of nodes $u$, such that $\{u,v\}$ is thin.
%Consider the induced subgraph of $G$ over $T\cup \{v\} \subseteq U \cup V$ (View these as a simply directed graph).
Let $\pset_l, \pset_r, \pset_c$ be the partial orders $\pset^{DL}, \pset^{DR}, \pset^C$ induced on set $\phi(T)$. According to \cref{lem:chain3_all_ordered}, each pair of segments $s, s' \in \phi(T)$, is comparable in at least one of the partial orders.
%and by~\Cref{lem:chain3_all_ordered}, form a partition for $S$.
We then have $|E(\pset_l) \cup E(\pset_r) \cup E(\pset_c)| \geq {|T| \choose 2}$.

Since for all $u \in T$, the edge $\{u,v\}$ is thin, there are at most $(k-2)$ DL-successors of $\phi(u)$ in $\phi(T)$, so we have $|E(\pset_l)| \leq (k-2)|T|$. Similarly, we have at most $(k-2)$ C-successors and $(k-2)$ DR-successors in $\phi(T)$. Therefore, $|E(\pset_r)|, |E(\pset_c)| \leq (k-2)|T|$. 
Combining the edges in the three partial orders, we get $|E(\pset_l) \cup E(\pset_r) \cup E(\pset_c)|\leq 3(k-2)|T|$. 
Putting these inequalities together, we have 
\[ |T|(|T|-1)/2 = {|T| \choose 2} \leq 3(k-2)|T|\]
which implies that $|T| \leq 6k-11<6(k-1).$
\end{proof}

\begin{theorem} \label{thm:chain3}
For all $n, m, k \in \N$,  $Z_{\chain^3}(m, n;k, k) \leq (3m+6n)(k-1)$.% in $\chain^3$ graphs.  
\end{theorem}
\begin{proof} Let $G = (U \cup V, E) \in \chain^3$ be a $K_{k, k}$-free graph, where $|U| = m$, $|V| = n$. 
    Note that we defined every edge in $E$ as either thin or bulky.
    According to \cref{lem:chain3_bulky}, the number of bulky edges in $E$ is at most $3(k-1) m$. According to \cref{lem:chain3_thin}, the number of thin edges in $E$ is at most $6(k-1) n$. This means that the total number of edges is $|E| \leq (3m+6n)(k-1)$.
\end{proof}

\begin{corollary} \label{cor:chain3_main}
     When $m = n$, we have $Z_{\chain^3}(n;k) \leq 9(k-1)n$.% in $\chain^3$ graphs. 
\end{corollary}
%Combining the previous two lemmas, the total number of edges in $G$ is at most $3(k-1)|\Se| + 6(k-1) |\Bo|$. In our symmetric scenario $|\Se| = |\Bo|$, this implies that $Z(n;k) \leq 9(k-1) n$. 

\dichotomyThm*
\begin{proof}
The first part will be proved in \cref{sec:chain3}. The second part is a simple corollary of existing results: Chazelle \cite{chazelle1990lower} constructed $K_{2,2}$-free $\prig$ graphs on $n'$ vertices that have at least $\Omega(n' \cdot \frac{\log n'}{\log \log n'})$ edges. Chan and Har-Peled~\cite{chan2023number} assert the applicability of this result in graph theory, since the original context was pointer machines. By a simple modification that we show in \cref{lem:lb_duplication}, there exist $K_{k,k}$-free $\prig$ graphs on $N$ vertices with at least $\Omega(N\cdot (k-1) \cdot \frac{\log \frac{N}{k-1}}{\log \log \frac{N}{k-1}}) = \Omega(N\cdot k \cdot \frac{\log N}{\log \log N})$ edges.
Since $\prig \subseteq \conv^2$ (\cref{lem:prig_in_conv2}), which is contained in $\chain^4$ \cite{chaplick2014intersection}, then the constructed graphs give a lower bound for Zarankiewicz numbers for $\chain^4$ graphs.
\end{proof}

\section{Grid Intersection Graphs} \label{sec:gig}

This section will use a more formal definition of grid intersection graphs.
We say that $(\Hor, \Ver, \phi)$ is a grid intersection representation of a bipartite graph $G = (U \cup V, E)$, if $\Hor$, $\Ver$ are sets of horizontal and vertical segments in $\mathbb{R}^2$ respectively, $\phi: U \mapsto \Hor, V \mapsto \Ver$, and for all $u\in U$, $v\in V$, we have $\{u, v\} \in E$ if and only if $\phi(u)$ and $\phi(v)$ intersect. A bipartite graph $G$ is a grid intersection graph ($\gig$) if it has a grid intersection representation. 
For a segment $\horver \in \Hor \cup \Ver$, let $N(\horver) = \{\horver' \in \Hor \cup \Ver: \horver' \cap \horver \neq \emptyset \}$, e.g the set of segments that intersect $\horver$. 

%For a bipartite graph $G$, we use $k(G)$ to denote the minimum positive integer such that $G$ does not have $K_{k, k}$ as a subgraph. When the context is clear, we omit the $G$ and just write $k$.

Let $k$ be the minimum positive integer such that $G$ does not have $K_{k, k}$ as a subgraph. %Whenever $k$ is used in this section, we are referring to the most recently defined graph $G$.

\begin{observation} \label{obs:degeneracy}
    If $G = (U \cup V, E) \in \gig$ is $27(k-1)$-degenerate, then $|E| \leq 27(k-1)(|U| + |V|)$.
\end{observation}

The proof is done via charging arguments. 

\subsection{Charging Arguments}
%Otherwise, there must exist a graph $G = (U \cup V, E) \in \gig$, such that $\forall w \in U \cup V$, $|N(w)| > 27k$. In this section, we will show that such a node always exists. We assume by contradiction that 

%For any node $w \in U \cup V$, let $N(\phi(w))$ be the set of segments $\phi(N(w))$, where $N(w)$ is the neighborhood of $w$. In other words, we use the neighborhood notation directly on the se
%For a segment $\horver \in \Hor \cup \Ver$ in the representation of a graph $G = (U \cup V, E)$, we use $N(\horver)$ to denote the se

For each segment $\horver \in \Hor \cup \Ver$, let $\horver.x$ and $\horver.y$ denote its projections to the $x$-axis (horizontal) and $y$-axis (vertical) respectively. Note that the projections are closed intervals and points.
%A horizontal segment $\hor \in \Hor$ projected to the $x$-axis is a closed interval, and projected to the $y$ axis, is a point.
For an interval $[a, b]$, let $\min([a, b]) = a$ and $\max([a, b]) = b$.
For two intervals $I$ and $I'$, we say that $I < I'$, if $\max(I) < \min(I')$.

For any segment $\horver \in \Hor \cup \Ver$, we define $\dir_{down}(\horver) = \{\horver' \in \Hor \cup \Ver: \horver'.y < \horver.y \} $, i.e. the set of segments whose projection to the $y$ axis is entirely smaller than that of $\horver$. Similarly, we define $\dir_{up}(\horver) = \{\horver' \in \Hor \cup \Ver: \horver'.y > \horver.y \} $, $\dir_{left}(\horver) = \{\horver' \in \Hor \cup \Ver: \horver'.x < \horver.x \} $, $\dir_{right}(\horver) = \{\horver' \in \Hor \cup \Ver: \horver'.x > \horver.x \} $. For a set $S$, we define $\dir_{up}(S) = \bigcap_{\horver \in S} \dir_{up}(\horver) $ and similarly for the other directions.
%and direction $x \in \{up, down, left, right\}$, we define $\dir_{x}(S) = \bigcap_{\horver \in S} \dir_{x}(\horver) $.

\begin{definition} \label{def:down_heavy}
    We say that a vertical segment $\ver \in \Ver$ is \textbf{down-heavy} with respect to $\hor \in \Hor$, if $\hor.y \subset \ver.y$ and $|N(\ver)\cap \dir_{down}(\hor)| \geq 3(k-1)$.
\end{definition}

\begin{definition} \label{def:up_heavy}
    We say that a vertical segment $\ver \in \Ver$ is \textbf{up-heavy} with respect to $\hor \in \Hor$, if $\hor.y \subset \ver.y$ and $|N(\ver)\cap \dir_{up}(\hor)| \geq 3(k-1)$.
\end{definition}

\begin{figure}[h]
    \centering
    \includegraphics[scale=0.15]{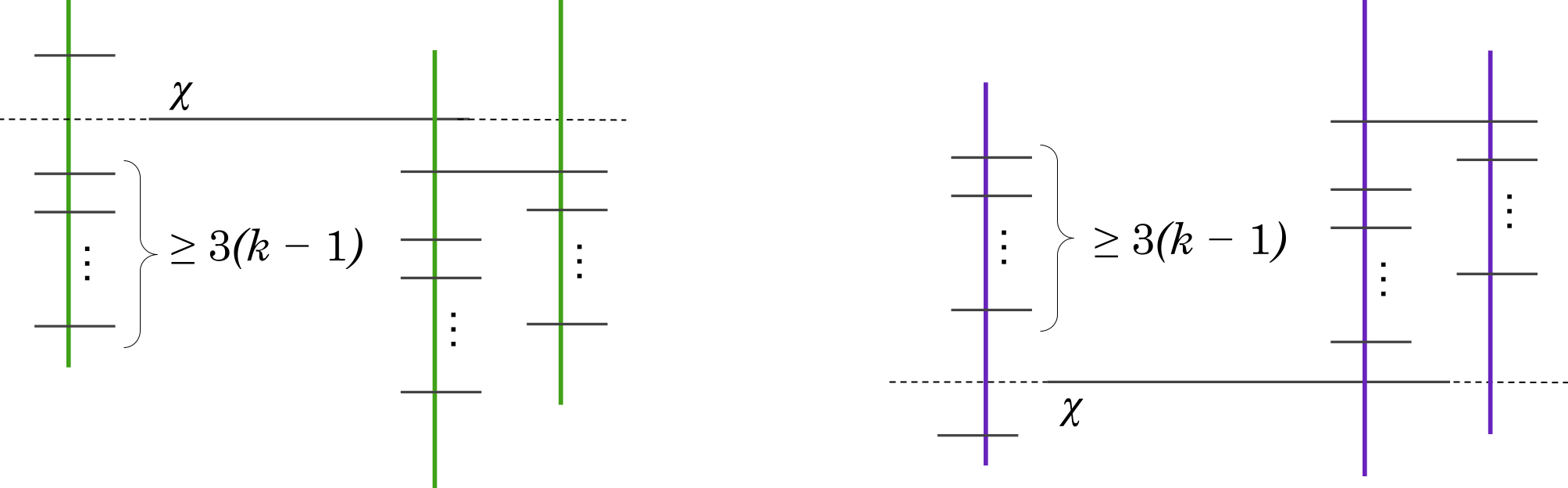}
    \caption{down-heavy segments are represented by the green vertical segments on the left and up-heavy segments by the purple ones on the right. }
    \label{fig:down_heavy}
\end{figure} 

To show that a node $w \in U \cup V$, such that $|N(w)| \leq 27(k-1)$ exists in any $\gig$ graph, we present
a payment scheme algorithm whereby each horizontal segment starts with $27(k-1)$ credits, which is subsequently paid to vertical segments. We will see that each vertical segment $\ver \in \Ver$ whose degree is at least $27(k-1)$, receives at least $|N(\ver)|$ credits.

\begin{algorithm}[H]
\DontPrintSemicolon
\caption{Neighbor payments}
\label{alg:close_type}
    \textbf{From each} $\hor\in \Hor$\;
    \Indp
    %\ForEach {$\hor\in \Hor$ }{
        \textbf{Pay} $\frac{9}{2}$ credits to each of the following vertical segments in $N(\hor)$:\;
        \Indp
        $k-1$ leftmost up-heavy vertical segments with respect to $\hor$\;
        $k-1$ rightmost up-heavy vertical segments w.r.t. $\hor$\;
        $k-1$ leftmost down-heavy vertical segments w.r.t. $\hor$\;
        $k-1$ rightmost down-heavy vertical segments w.r.t. $\hor$\;
        \Indm
    \Indm
    %}
\end{algorithm}

\begin{algorithm}[H]
\DontPrintSemicolon
\caption{Further payments}
\label{alg:estranged_type}
    \textbf{From each} $\hor\in \Hor$\;
    \Indp
    %\ForEach {$\hor\in \Hor$ }{
        \textbf{Pay} $\frac{9}{4}$ credits to each of the following segments $\ver$, where $\hor.y \subset \ver.y$:\;
        \Indp
        $k-1$ rightmost up-heavy vertical segments in $\dir_{left}(\hor)$\;
        $k-1$ rightmost down-heavy vertical segments in $\dir_{left}(\hor)$\;
        $k-1$ leftmost up-heavy vertical segments in $\dir_{right}(\hor)$\;
        $k-1$ leftmost down-heavy vertical segments in $\dir_{right}(\hor)$\;
        \Indm
    \Indm
    %}
\end{algorithm}

\begin{figure}[h]
    \centering
    \includegraphics[scale=0.16]{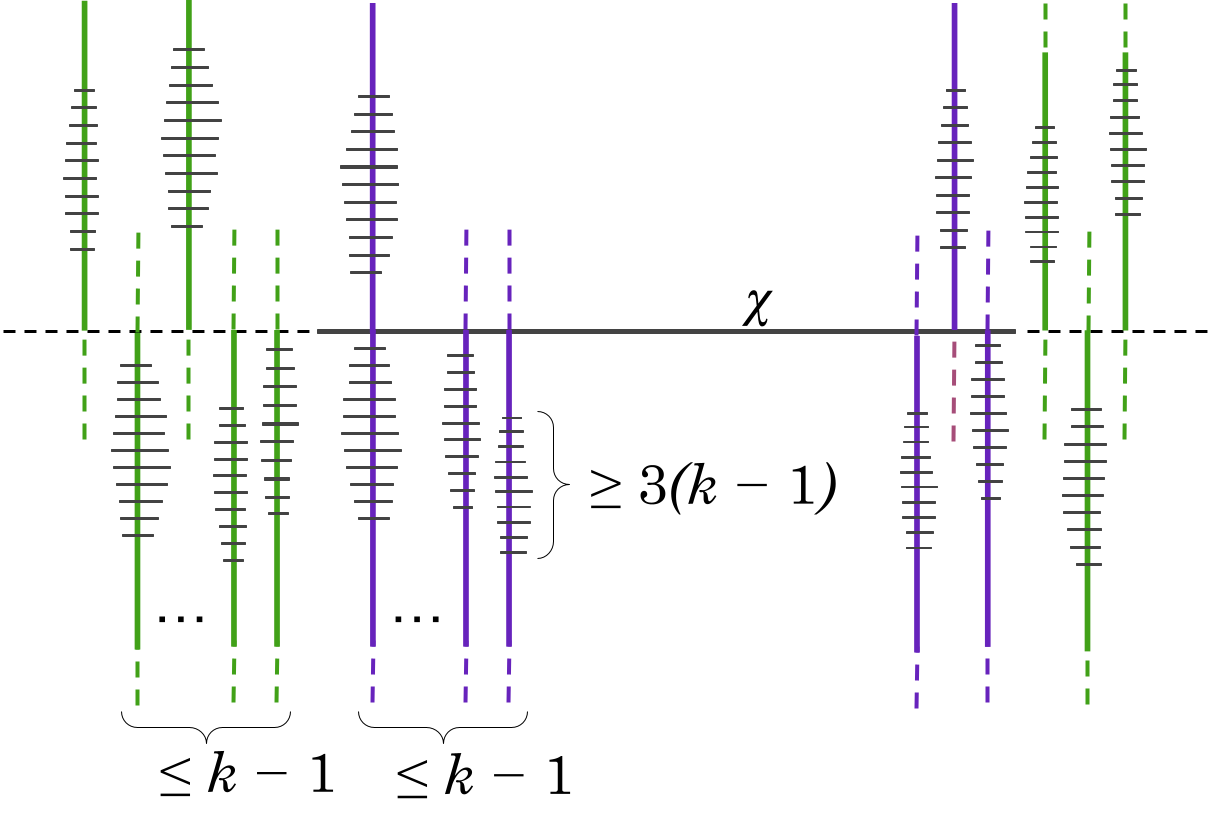}
    \caption{\Cref{alg:close_type} credit receivers are shown in purple and \Cref{alg:estranged_type} credit receivers in green.}
    \label{fig:GIG_deg_algo_credits}
\end{figure} 

Now, we will analyze the credits received by an arbitrary vertical segment $\ver \in \Ver$.  %We say $\hor \in N(\ver)$ is \textbf{generous to $\ver$} (or simply generous), if $\cref{alg:close_type}$ pays any credits to $\ver$ from $\hor$. The remaining segments in $N(\ver)$ we consider \textbf{stingy} to $\ver$.

For a set of horizontal segments $S \subseteq \Hor$, let $S.y$ denote the interval $[min, max]$, where $min = \min_{\hor \in S}(\hor.y)$ and $max = \max_{\hor \in S}(\hor.y)$.
\begin{lemma} \label{lem:payments_from_interval}
    Let $S \subset N(\ver)$, such that 
    \begin{itemize}
        \item $|S| = 3(k-1)$
        \item $|N(\ver) \cap \dir_{up} (S)| \geq 3(k-1)$ (Intuition: $S$ is not among the highest elements of $N(\ver)$)
        \item $|N(\ver) \cap \dir_{down} (S)| \geq 3(k-1)$ (Intuition: $S$ is not among the lowest elements of $N(\ver)$)
        %\item $N(\ver) \subseteq \dir_{up} (S) \cup \dir_{down} (S) \cup S $ ($S$ contains consecutive neighbors of $\ver$)
    \end{itemize}
    Then at least $\frac{9}{2} (k-1)$ credits are paid to $\ver$ from horizontal segments $\hor \in \Hor$, such that $\hor.y \in S.y$.
\end{lemma}
The proof is deferred to $\cref{sec:gig_appendix}$.

\begin{lemma} \label{lem:vertical_credits}
    If $|N(\ver)| \geq 27(k-1)$, then $\ver$ receives at least $|N(\ver)|$ credits.
\end{lemma}
\begin{proof}
    We choose the maximum number $\ell$ of disjoint sets $S_1, \ldots S_{\ell}$ in $N(\ver)$, such that every set $S_i$ satisfies the conditions of \cref{lem:payments_from_interval}. Since each set receives $\frac{9}{2} (k-1)$ credits from a distinct set of horizontal segments, then the total amount in credits $\ver$ receives is at least $\frac{9}{2} (k-1) \ell$. %The the case that the degree of $\ver$ is minimal, we have $\ell \geq \lfloor \frac{27(k-1) - 3(k-1) - 3(k-1)}{3(k-1)}\rfloor  \geq 7$.

    It remains to see that $\frac{9}{2} (k-1) \ell \geq |N(\ver)|$. Note that $\ell \geq \lfloor \frac{|N(\ver)| - 3(k-1) - 3(k-1)}{3(k-1)}\rfloor \geq \frac{|N(\ver)| - 9(k-1)}{3(k-1)}$. We get that $\frac{9}{2} (k-1) \ell \geq \frac{9}{2} (k-1) \frac{|N(\ver)| - 9(k-1)}{3(k-1)} = \frac{3}{2}|N(\ver)| - \frac{27}{2}(k-1) = |N(\ver)| + \frac{1}{2}(|N(\ver)| - 27(k-1))$. Since according to our assumption $|N(\ver)| \geq 27(k-1)$, then this result is at least $|N(\ver)|$. 
\end{proof}

\begin{lemma} \label{lem:gig_degenerate}
    Any $G = (U \cup V, E) \in \gig$ is $27(k-1)$-degenerate.
\end{lemma}
\begin{proof}
    Assume by contradiction that there is a graph $G = (U \cup V, E) \in \gig$, where no such node exists.
    %Let $(\Hor, \Ver, \phi)$ be a grid representation of $G$. If $|\Hor| \geq |\Ver|$, we rotate the representation $90$ degrees and get a representation, where $|\Hor| \leq |\Ver|$. Therefore, let us assume that $(\Hor, \Ver, \phi)$ is a representation where $|\Hor| \leq |\Ver|$. 
    Since according to our assumption, for any node $w \in U \cup V$, we have $|N(w)| \geq 27(k-1) + 1$, then by running \cref{alg:close_type} and \cref{alg:estranged_type}, we see that for each horizontal node, fewer credits are given out than their degree. This implies that the amount in credits given by these algorithms is smaller than $|E|$.

    For each vertical node $\ver \in \Ver$, on the other hand, according to \cref{lem:vertical_credits}, the amount in credits received is at least that of their degree. This implies that the total amount in credits received from these algorithms greater than or equal to $|E|$. This is a contradiction.
\end{proof}

\begin{theorem} \label{thm:gig}
For all $n, m, k \in \N$,  $Z_{\gig}(m,n;k, k) \leq 27(m+n)(k-1)$.  
\end{theorem}
\begin{proof}
    The result follows from \cref{lem:gig_degenerate} and \cref{obs: degenrate counting}.
\end{proof}

%\vspace{1em} 
%\textbf{\Cref{thm:intro_GIG}}:
%We have $Z(n;k) \leq 54n(k-1)$ for grid intersection graphs.   
\introGIG*
\begin{proof}
    The result follows from \cref{thm:gig}, when $n = m$
\end{proof}

\bibliography{ref}

\appendix

\section{Higher Ferrers Dimension}
\label{sec:high dim}
We prove~\cref{cor: high dim} in this section. Note that the proof closely follows Section 2.2 of Chan and Har-Peled~\cite{chan2023number}. 
We repeat their definitions here for completeness. For integers $a \leq b$, we denote $\{a,a+1,\ldots, b\}$ by $[[a: b]]$. 
A \textbf{dyadic range} is an integer range of the form $[[s 2^i: (s+1)2^{i+1}-1]]$ for some integers $s$ and $i$.
Considering a $\chain$ graph, we can view the points as integers $\pset = \{1,\ldots, n\}$. 
Denote by $\dset(n)$ the set of all dyadic ranges. 
Each integer (point) appears in at most  $\lceil \log n\rceil$ dyadic ranges.

Notice that a ray is simply a range $[[a, n]]$ for some $a$.  

\begin{lemma}[An analogue of Lemma 2.2 in~\cite{chan2023number}]
Let $r$ be a ray. Denote by $\sf{dy}(r)$ the set of minimal disjoint dyadic ranges covering $r$. Then $\sf{dy}(r)$ is unique and contains at most $\lceil \log n \rceil$ ranges.     
\end{lemma}

Remark that we have $\lceil \log n\rceil$ instead of $2 \lceil \log n\rceil$ as in~\cite{chan2023number} since we have rays instead of intervals. The following lemma implies~\cref{cor: high dim}. 

\begin{lemma}
\label{lem: dyadic decomp}
Let $G = (A \cup B, E) \in \chain^d$ for $d \geq 3$ that does not contain $K_{k,k}$. Then the number of edges in $G$ is at most $O(|A|+|B|) \cdot k \cdot \lceil \log n\rceil^{d-3}$ where $n = \max \{|A|,|B|\}$.     
\end{lemma}
\begin{proof}
The proof is obtained by induction on $d$. The base case when $d = 3$ is proved in~\cref{thm:chain3}. 
When $d > 3$, assume the claim is true for all dimensions up to $d-1$. We write $G = G' \cap \widehat{G}$ where $G' \in \chain$ and $\widehat{G} \in \chain^{d-1}$. 
Since $G' \in \chain$, it is an intersection bigraph of rays $\rset$ and points $\pset$ where $a \in A$ is mapped to $r_a \in \rset$ and $b \in B$ to $p_b \in \pset$. Assume w.l.o.g. that $\pset = \{1,\ldots, q\}$ for $q \leq n$.

%For ray $r$, denote by $\iset(r) \subseteq \dset(q)$ the set of minimal dyadic ranges that cover $r$ guaranteed by Lemma~\ref{lem: dyadic decomp}. 
For each dyadic range $\nu \in \dset(q)$, we have an auxiliary graph $G_{\nu}$ which  is an induced subgraph $G[A_{\nu} \cup B_{\nu}]$ where  $A_{\nu} = \{a \in A: \nu \in \sf{ dy}(r_a)\}$ and $B_{\nu} = \{b: p_b \in \nu \}$. 
Notice that $|E(G)| \leq \sum_{\nu \in \dset(n)} |E(G_{\nu})|$. 

Moreover, $G_{\nu} = G'[A_{\nu} \cup B_{\nu}] \cap \widehat{G}[A_{\nu} \cup B_{\nu}]$ where $G'[A_{\nu} \cup B_{\nu}]$ is a biclique; therefore $G_{\nu} \in \chain^{d-1}$ and is free of $k$-by-$k$ biclique.
This allows us to invoke the induction hypothesis, which implies that 
\[ |E(G)| \leq \sum_{\nu \in \dset(n)} O(|A_{\nu}| + |B_{\nu}|)k \lceil \log n\rceil^{d-4}\]
Finally, since each vertex $a \in A$ appears in at most $\lceil \log q \rceil \leq \lceil \log n\rceil$ sets $A_{\nu}$ (and similarly for each vertex $b \in B$), this implies that the above sum is at most $O(|A|+ |B|)k \lceil \log n\rceil^{d-3}$. 
\end{proof}

\section{Intersection of Two Convex Graphs} \label{sec:conv2}

In this section, we will use a more formal definition of $\prig$ graphs.
We say that $(\Hor, \Ver, \phi)$ is a point-rectangle intersection representation of a graph $G = (U \cup V, E)$, if $\Hor$, $\Ver$ are sets of points and rectangles, respectively, in $\mathbb{R}^2$, $\phi: U \mapsto \Hor, V \mapsto \Ver$, and for all $u\in U$, $v\in V$, we have $\{u, v\} \in E$ if and only if $\phi(u)$ and $\phi(v)$ intersect. A graph $G = (U \cup V, E)$ is a point-rectangle intersection bigraph, if it has a point-rectangle intersection representation. We call the class of such graphs $\prig$.

\begin{figure}[h]
    \centering
    \includegraphics[scale=0.12]{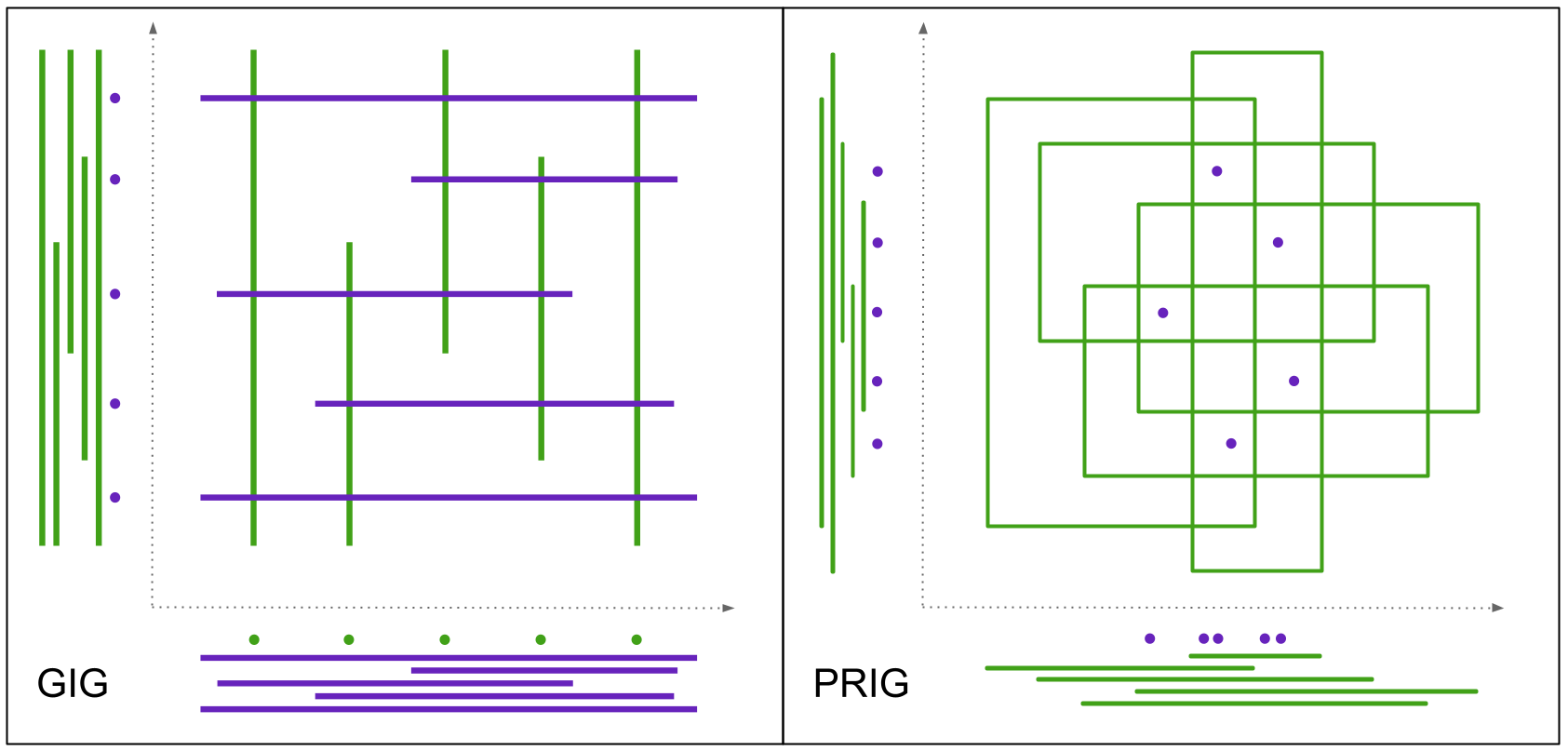}
    \caption{Examples of $\gig$ and $\prig$ graphs with their projections into axes}
    \label{fig:gig_prc}
\end{figure}

\begin{lemma} \label{lem:conv2_is_gig_prig}
    Any graph $G = (U \cup V, E) \in \conv^2$ is a disjoint union of $\gig$ and $\prig$ graphs.
\end{lemma}
\begin{proof}
Since $G \in \conv^2$, then there exist graphs $G_1$ and $G_2$, such that $G_1 \cap G_2 = G$. Let $V_1$ and $V_2$ be the partitions that the graphs $G_1$ and $G_2$, respectively, are convex over, and $U_1 \subseteq V(G_1)$, $U_2 \subseteq V(G_2)$, the other partitions.

Since $\conv$ graphs are closed under taking induced subgraphs, we can assume that $U_1 \cup V_1 = U_2 \cup V_2$. Let $A_1 = U_1 \cap U_2$, $A_2 = U_1 \cap V_2$, $B_1 = V_1 \cap V_2$, $B_2 = V_1 \cap U_2$, as shown on \cref{fig:partition_intersections}.

\begin{figure}[h]
    \centering
    \includegraphics[scale=0.1]{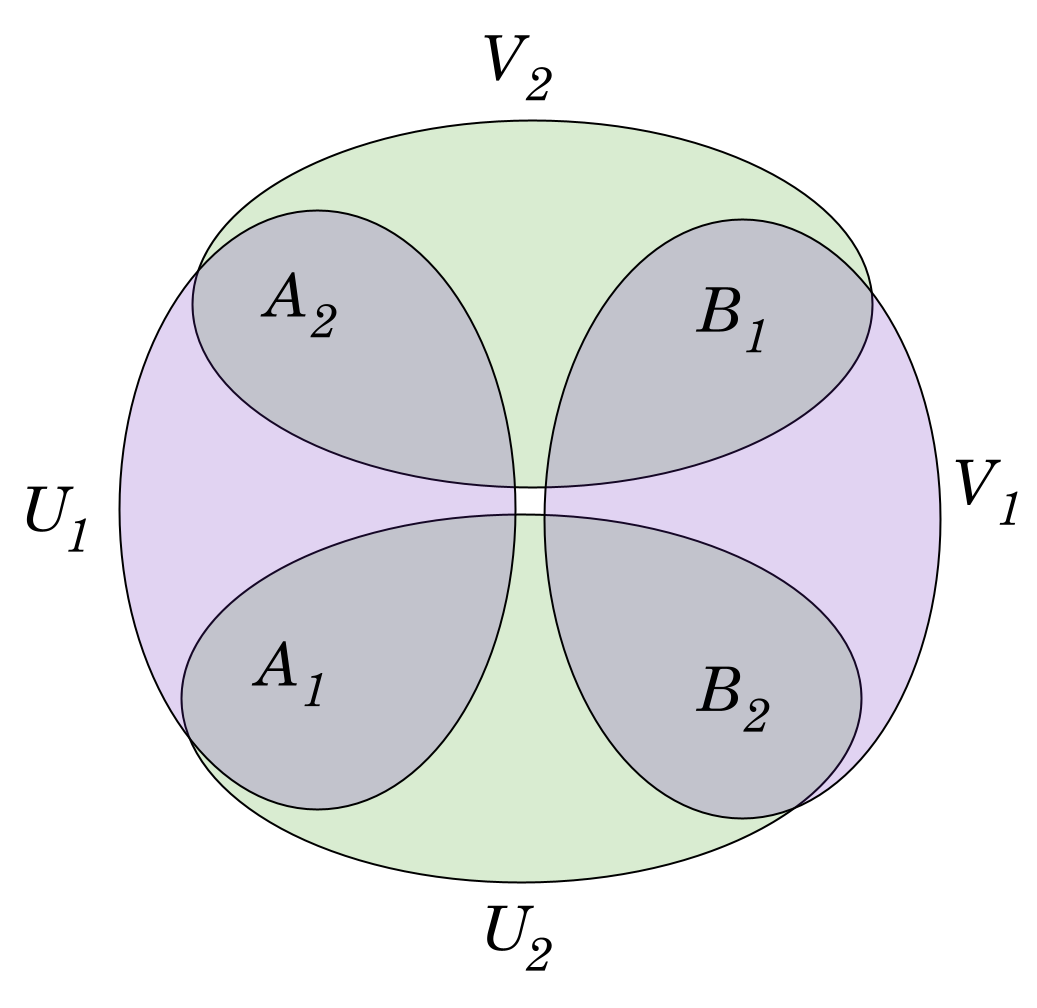}
    \caption{}
    \label{fig:partition_intersections}
\end{figure}

We have $U \cup V = A_1 \cup A_2 \cup B_1 \cup B_2$ and all edges must start and end at these nodes. Since $A_1$ is in the same partition as $A_2$, and $B_2$, in at least one of the graphs, there can be no edges between $A_1$ and $A_2 \cup B_2$. The same is true for $B_1$. Clearly, $G$ can have one component consisting of nodes $A_1 \cup B_1$ and another component consisting of nodes $A_2 \cup B_2$, with no edges between these components.

It remains to see that $G[A_1 \cup B_1] \in \prig$ and $G[A_2 \cup B_2] \in \gig$. To do this, we show that there exist representations of these graphs in $\prig$ and $\gig$ respectively, such that the projections of these representations to $x$ and $y$ axis are the convex representations of graphs $G_1$ and $G_2$ induced on $A_1 \cup B_1$ and $A_2 \cup B_2$ respectively.

Let $H_1 = G_1[A_1 \cup B_1]$ and $H_2 = G_2[A_1 \cup B_1]$. Since $G_1$ and $G_2$ are convex over $V_1$ and $V_2$, and $B_1 =  V_1 \cap V_2$, then $H_1$ and $H_2$ are convex over $B_1$.
Let $(\Hor_1, \Ver_1, \phi_1)$ be the point-segment representation of $H_1$, such that $\Ver_1$ are the points. Then $\phi_1: A_1 \mapsto \Hor_1, B_1 \mapsto \Ver_1$. Let $(\Hor_2, \Ver_2, \phi_2)$ be the point-segment representation of $H_2$, such that $\Ver_2$ are the points. Then $\phi_2: A_1 \mapsto \Hor_2, B_1 \mapsto \Ver_2$. We construct a mapping $\phi: A_1 \mapsto \Hor, B_1 \mapsto \Ver$, where $\Hor$ is a set of rectangles, and $\Ver$ a set of points in $\R^2$.
For each $w \in A_1 \cup B_1$,  $\phi(w)$ is the object whose projection to the $x$-axis is $\phi_1(w)$ and whose projection to the $y$-axis is $\phi_2(w)$. Note that for all $a \in A_1$, $\phi(a)$ is a rectangle, and for all $b \in B_1$, $\phi(b)$ is a point. It is easy to verify that $\phi(a)$ intersects $\phi(b)$, if and only if $\phi_1(a)$ intersects $\phi_1(b)$ and $\phi_2(a)$ intersects $\phi_2(b)$, which implies that there is an edge $\{a, b\}$ in the graph represented by $\phi$, if and only if $\{a, b\} \in E(H_1)\cap E(H_2)$. This means that $(\Hor, \Ver, \phi)$ is a valid representation of $G[A_1 \cup B_1] \in \prig$.

Showing that $G[A_2 \cup B_2] \in \gig$ is similar, but not quite symmetric. Let $H_1 = G_1[A_2 \cup B_2]$ and $H_2 = G_2[A_2 \cup B_2]$. Since $G_1$ and $G_2$ are convex over $V_1$ and $V_2$, and $A_2 \subseteq V_2$, $B_2 \subseteq V_1$, then $H_1$ and $H_2$ are convex over $A_2$ and $B_2$ respectively.
Let us define $(\Hor_1, \Ver_1, \phi_1)$ to be the point-segment representation of $H_1$, such that $\Hor_1$ are the points. Then $\phi_1: A_2 \mapsto \Hor_1, B_2 \mapsto \Ver_1$. Let $(\Hor_2, \Ver_2, \phi_2)$ be the point-segment representation of $H_2$, such that $\Ver_2$ are the points. Then $\phi_2: A_2 \mapsto \Hor_2, B_2 \mapsto \Ver_2$. We construct a mapping $\phi: A_2 \mapsto \Hor, B_2 \mapsto \Ver$, where $\Hor$ is a set of horizontal segments, and $\Ver$ is a set of vertical segments in $\R^2$.
\begin{figure}[h]
    \centering
    \includegraphics[scale=0.12]{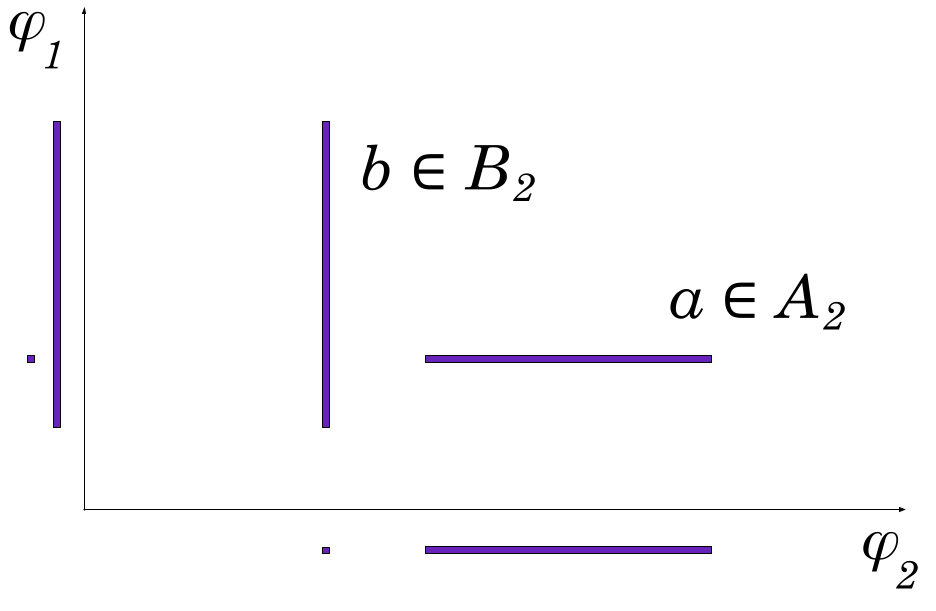}
    \caption{}
    \label{fig:conv2_GIG_proof}
\end{figure}
For each $w \in A_2 \cup B_2$,  $\phi(w)$ is the object whose projection to the $y$-axis is $\phi_1(w)$ and whose projection to the $x$-axis is $\phi_2(w)$. Note that for all $a \in A_2$, $\phi(a)$ is a horizontal segment and for all $b \in B_1$, $\phi(b)$ is a vertical segment. It is easy to verify that $\phi(a)$ intersects $\phi(b)$, if and only if $\phi_1(a)$ intersects $\phi_1(b)$ and $\phi_2(a)$ intersects $\phi_2(b)$, which implies that there is an edge $\{a, b\}$ in the graph represented by $\phi$, if and only if $\{a, b\} \in E(H_1)\cap E(H_2)$. This means that $(\Hor, \Ver, \phi)$ is a valid representation of $G[A_1 \cup B_1] \in \gig$.

\end{proof}

For graphs $G_1$, $G_2$, let $G_1 \dist G_2$ to denote the disjoint union of $G_1$ and $G_2$.

\begin{lemma} \label{lem:conv2_disjoint_union}
    $\conv^2$ is closed under disjoint union.
\end{lemma}
\begin{proof}
Let $G$, $G'$ be any disjoint graphs in $\conv^2$. Then we have graphs $G_1, G_2 \in \conv$, such that $G_1 \cap G_2 = G$, and graphs $G'_1, G'_2 \in \conv$, such that $G'_1 \cap G'_2 = G'$. Since $\conv$ is closed under disjoint union, then $G_1 \dist G'_1 \in \conv$ and $G_2 \dist G'_2 \in \conv$. Then $(G_1 \dist G'_1 )\cap (G_2 \dist G'_2) \in \conv^2$. Since $G_1$ only shares vertices with $G_2$, and similarly, $G'_1$ only shares vertices with $G'_2$, then $(G_1 \dist G'_1 )\cap (G_2 \dist G'_2) = (G_1 \cap G_2) \dist (G'_1 \cap G'_2) = G \dist G' \in \conv^2$.

\end{proof}

\begin{lemma} \label{lem:prig_in_conv2}
    $\prig \subseteq \conv^2$.
\end{lemma}
\begin{proof}
Let $G = (U \cup V, E)$ be any graph in $\prig$, and $(\Hor, \Ver, \phi)$ the representation of $G$ in $\prig$, such that $\Hor$ is a set of points and $\Ver$ a set of rectangles in $\R^2$. Then we can construct mappings $\phi_1$ and $\phi_2$, such that for all $w \in V(G)$, $\phi_1(w)$ and $\phi_2(w)$ are the projections of $\phi(w)$ to $x$ and $y$ axis respectively. Let $G_1$ and $G_2$ be the graphs represented by $\phi_1(w)$ and $\phi_2(w)$ respectively. Since $\phi_1(w)$ maps the vertices $V(G)$ to points and segments based on their partitions, then $G_1\in \conv$. Similarly for $G_2$. Since for any $u \in U$, $v \in V$, we have that $\phi(u)$ intersects $\phi(v)$ if and only if $\phi_1(u)$ intersects $\phi_1(v)$ and $\phi_2(u)$ intersects $\phi_2(v)$, then $G_1 \cap G_2 = G$. This proves that $G \in \conv^2$.
\end{proof}

\begin{lemma} \label{lem:gig_in_conv2}
    $\gig \subseteq \conv^2$.
\end{lemma}
\begin{proof}
The proof is symmetric to \cref{lem:prig_in_conv2}.
\end{proof}

\textbf{\Cref{thm:conv2}}:
    Any graph $G = (U \cup V, E) \in \conv^2$ if and only if it is a disjoint union of $\gig$ and $\prig$ graphs.

This follows from $\cref{lem:conv2_is_gig_prig}$, $\cref{lem:conv2_disjoint_union}$, $\cref{lem:prig_in_conv2}$, and $\cref{lem:gig_in_conv2}$.

\section{Segment Bottomless Rectangle Containment Bigraph} \label{sec:chain3_is_bot}

We say that $(\Hor, \Ver, \phi)$ is an interval containment representation of a graph $G = (U \cup V, E)$, if $\Hor$, $\Ver$ in $\mathbb{R}$ are intervals, $\phi: U \mapsto \Hor, V \mapsto \Ver$, and for all $u\in U$, $v\in V$, we have $\{u, v\} \in E$ if and only if $\phi(u)$ is contained within $\phi(v)$. A graph $G$ is an interval containment bigraph, if it has an interval containment representation. According to \cite{chaplick2014intersection}, the class of interval containment bigraphs is $\chain^2$.

\begin{proposition} \label{lem:int_cont_symmetric}
    If a graph $G = (U \cup V, E) \in \chain$ has a point ray representation $(\Hor, \Ver, \phi)$, $\phi: U \mapsto \Hor, V \mapsto \Ver$, then it also has a point ray representation $(\Hor', \Ver', \phi')$, $\phi': V \mapsto \Hor', U \mapsto \Ver'$, where $\Hor, \Hor'$ are points and $\Ver, \Ver'$ are rays.
\end{proposition}
\begin{proof}
    We are given $G = (U \cup V, E) \in \chain$ with a point ray representation $(\Hor, \Ver, \phi)$, $\phi: U \mapsto \Hor, V \mapsto \Ver$, where $\Hor$ represents points and $\Ver$ represents rays extending to $\infty$ in $\R$.% we will see that there is also a mapping $\phi': V \mapsto \Hor', U \mapsto \Ver'$, where $\Hor'$ represents points and $\Ver'$ represents rays extending to $\infty$ in $\R$.

    For $\ver \in \Ver$, let $\ver.p$ denote the minimum point in ray $\ver$. Then for all $u\in U$, $v \in V$, we say that $\phi(v).p < \phi(u)$, if and only if $\{u, v\} \in E$. 
    
    We construct $(\Hor', \Ver', \phi')$ such that for each $u \in U$, $\phi'(u) = (-\infty, \phi(u)]$, $\phi'(v) = \phi(v).p$. Then for all $u\in U$, $v \in V$, $\phi'(u)$ and $phi'(v)$ intersect if $\phi'(v) < \phi'(u).p$, which happens if and only if $\phi(v).p < \phi(u)$. This implies that $(\Hor', \Ver', \phi')$ is a representation of $G$, where the partition $U$ is mapped to rays and $V$ to points. By flipping the representation, it is easy to see that there is an equivalent representation where rays extend to $\infty$.
\end{proof}

\textbf{\cref{lem:chain3_is_brc}:}
A graph $G$ is of Ferrers dimension three ($G \in \chain^3$) if and only if it is a bottomless rectangle segment containment graph.  

\begin{proof}

Let $(\Se, \Bo, \phi)$ be a segment bottomless rectangle representation of graph $G = (U \cup V, E)$, such that $\phi: U \mapsto \Se, V \mapsto \Bo$. We will see that $G \in \chain^3$.

First, we consider the graph $G_x = (U \cup V, E_x)$ represented by the projections of $\Se, \Bo$ to the $x$ axis. Since $G_x$ is an interval containment bigraph, then $G_x \in \chain^2$ (\cite{saha2014permutation}, \cite{chaplick2014intersection}). Secondly, consider the graph $G_y = (U \cup V, E_y)$ represented by the projections of $\Se, \Bo$ to the $y$ axis. Since $G_y$ is an intersection graph of points and rays, then $G_y \in \chain$. Since each $b \in \Bo$ contains $s \in \Se$ if and only if $s.x \subseteq b.x$ and $ s.y \subseteq b.y$, then $G = G_x \cap G_y \in \chain^2 \cdot \chain = \chain^3$.

Next, we will show that any graph $G = (U \cup V, E) \in \chain^3$ has a segment bottomless rectangle representation. According to \cite{chaplick2014intersection}, there exist graphs $G_x = (U \cup V, E_x)\in \chain^2$ and $G_y = (U \cup V, E_y)\in \chain$, such that $G_x \cap G_y = G$. Let $(\Hor_x, \Ver_x, \phi_x)$ be an interval containment representation of $G_x$, such that $\Hor$ represent the contained segments and $\Ver$ the containing ones. Let us consider the case that $\phi_x: U \mapsto \Hor_x, V \mapsto \Ver_x$, the case that $\phi_x: V \mapsto \Hor_x, U \mapsto \Ver_x$ is symmetric. 
Let $(\Hor_y, \Ver_y, \phi_y)$ be a point ray representation of $G_y$, such that $\Hor_y$ are points and $\Ver_y$ are rays,  $\phi_y: U \mapsto \Hor_y, V \mapsto \Ver_y$ (\cref{lem:int_cont_symmetric}).

We construct a segment bottomless rectangle representation $(\Se, \Bo, \phi)$, $\phi: U \mapsto \Se, V \mapsto \Bo$, such that for all $w \in U \cup V$, $\phi(w).x = \phi_x(w)$, $\phi(w).y = \phi_y(w)$. Let $G'$ be the graph represented by $(\Se, \Bo, \phi)$. For all $u \in U$, $v \in V$, $\phi(u)$ is contained within $\phi(v)$ if and only if $\phi_x(u)$ is contained within $\phi_x(v)$ and $\phi_y(u)$ is contained within $\phi_y(v)$. This implies that $G' = G_x \cap G_y$, which means that $(\Se, \Bo, \phi)$ is a representation of $G$.

\end{proof}

\section{Grid Intersection Graphs} \label{sec:gig_appendix}

We say $\hor \in N(\ver)$ is \textbf{generous to $\ver$} (or simply generous), if $\cref{alg:close_type}$ pays any credits to $\ver$ from $\hor$. We consider the remaining segments in $N(\ver)$ \textbf{stingy} to $\ver$.

First, we will see that if for a section of a vertical node, not enough credits are gained from \cref{alg:close_type}, then there are horizontal segments in that section, which will contribute to \cref{alg:estranged_type} instead.

\begin{lemma} \label{lem:gig_box}
    Let $S \subseteq N(\ver)$, and $S' \subseteq S$ be the stingy subset, such that 
    \begin{itemize}
        \item $|S| \geq 3(k-1)$
        \item $|S - S'| \leq k-1$
        \item $N(\ver) \subseteq \dir_{up} (S) \cup \dir_{down} (S) \cup S $ ($S$ contains consecutive neighbors of $\ver$)
    \end{itemize}
    Then there exists a set $S^* \in \dir_{right}(\ver)$, such that
    \begin{itemize}
        \item $|S^*| \geq k-1$
        \item $S^*.y \subseteq S.y$
        \item $\bigcap_{\hor \in S' \cup S^*} \hor.x  \neq  \emptyset$
    \end{itemize}
\end{lemma}
\begin{proof}
    The number of stingy nodes is at least $2(k-1)$. Let the stingy subset of $S$ be $S'$. For any $\hor \in S'$, there are at least $k-1$ segments in either $\dir_{up}(\hor)\cap S'$ or $\dir_{down}(\hor)\cap S'$.
    
    Let $\hor \in S'$ be the segment with the smallest $\max(\hor.x)$. Let $\dir_{down}(\hor)\cap S' = S'_{down}$ and $\dir_{up}(\hor)\cap S' = S'_{up}$. We consider the case that $|S'_{down}| \geq k-1$. The case that $|S'_{up}| \geq k-1$ is symmetric. Let $Q$ be the $(k-1)$ rightmost down-heavy vertical segments in \cref{alg:close_type}, that are each paid $\frac{9}{2}$ credits from $\hor$. Let $\ver' \in Q$ be the vertical segment with the largest $\min(\ver'.y)$. 

    Since $\ver' \in Q$ has the largest $\min(\ver'.y)$ and $\hor \in S'$ has the smallest $\max(\hor.x)$, then $\{\hor\} \cup S'_{down} \cap N(v')$ and $Q \cup \{\ver\}$ induce a biclique (\cref{fig:GIG_degree_RD}). Since $|Q \cup \{\ver\}| = k$ and there cannot exist a $K_{k, k}$, then $|S'_{down} \cap N(v')| \leq k-2$, which implies that $\min(\ver'.y)$ is greater than $\min(S.y)$ ($\ver'$ does not extend far enough down to include $k-1$ stingy nodes below $\hor$).

    \begin{figure}[h]
        \centering
        \includegraphics[scale=0.18]{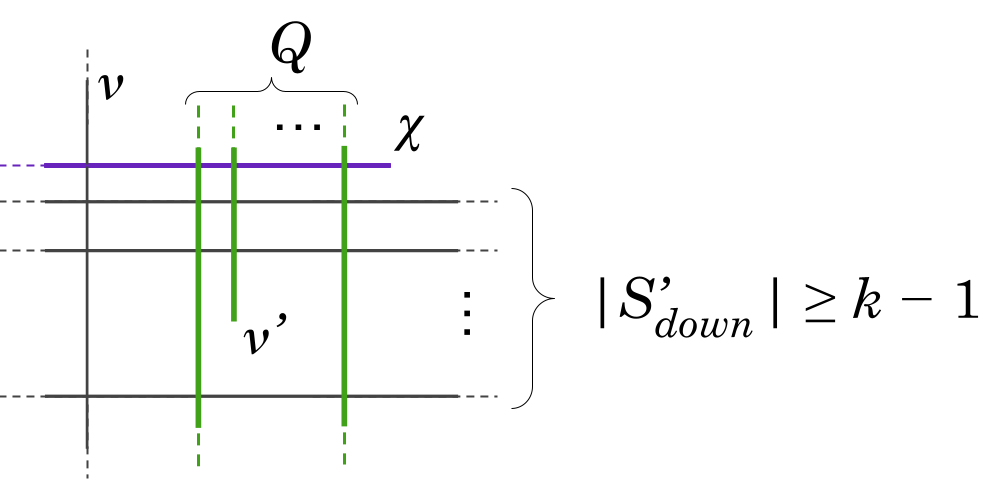}
        \caption{ Depicted are the stingy horizontal segments below $\ver$, the vertical segments $Q$ (green), and $\ver$.}
        \label{fig:GIG_degree_RD}
    \end{figure} 

    Since $\ver'$ is down-heavy w.r.t $\hor$, then $|N(\ver') \cap \dir_{down}(\hor)| \geq 3(k-1)$ (\cref{def:down_heavy}). Let $\Gamma = N(\ver') \cap \dir_{down}(\hor)$. Note that  $\Gamma.y \subseteq S.y$, because $\min(\ver'.y)$ is greater than $\min(S.y)$ and $\max(\Gamma.y) < \hor.y$. Due to how $S$ was chosen (the last condition), this implies that the nodes in $\Gamma$ that intersect $\ver$ are in $S$.

    Since $(k-1)$ nodes of $S'_{down} = \dir_{down}(\hor)\cap S'$ cannot be in $N(\ver')$, then $|\Gamma \cap S'| \leq k-2$. Since $|S - S'| \leq k-2$, then $|\Gamma - S| \geq 3(k-1) - (k-2) - (k-2) = k+1$.

    Let us consider the conditions for $S^* = \Gamma - S$.
    Since $\Gamma \in N(\ver')$, $\ver' \in Q \subseteq \dir_{right}(\ver)$, and none of $\Gamma - S$ intersect $\ver$, then $\Gamma - S \in \dir_{right}(\ver)$. We already saw that $|\Gamma - S| \geq k-1$ and $\Gamma.y \subseteq S.y$. Since $\hor \in S'$ is the segment with the smallest $\max(\hor.x)$, then all intervals in $S'$ projected to the $x$ axis must include $\ver'.x$. We have $\ver'.x \in \bigcap_{\hor \in S' \cup \Gamma} \hor.x$, which fulfills the last condition.

\end{proof}

\begin{lemma} \label{lem:gig_box_left}
    Let $S \subseteq N(\ver)$, and $S' \subseteq S$ be the stingy subset, such that 
    \begin{itemize}
        \item $|S| \geq 3(k-1)$
        \item $|S - S'| \leq k-1$
        \item $N(\ver) \subseteq \dir_{up} (S) \cup \dir_{down} (S) \cup S $ ($S$ contains consecutive neighbors of $\ver$)
    \end{itemize}
    Then there exists a set $S^* \in \dir_{left}(\ver)$, such that
    \begin{itemize}
        \item $|S^*| \geq k-1$
        \item $S^*.y \subseteq S.y$
        \item $\bigcap_{\hor \in S' \cup S^*} \hor.x  \neq  \emptyset$
    \end{itemize}
\end{lemma}
\begin{proof}
    The proof is symmetric to \cref{lem:gig_box}.
\end{proof}

\vspace{0.5cm}
\textbf{\cref{lem:payments_from_interval}}: 
    Let $S \subset N(\ver)$, such that 
    \begin{itemize}
        \item $|S| \geq 3(k-1)$
        \item $|N(\ver) \cap \dir_{up} (S)| \geq 3(k-1)$ ($S$ is not among the highest elements of $N(\ver)$)
        \item $|N(\ver) \cap \dir_{down} (S)| \geq 3(k-1)$ ($S$ is not among the lowest elements of $N(\ver)$)
        \item $N(\ver) \subseteq \dir_{up} (S) \cup \dir_{down} (S) \cup S $ ($S$ contains consecutive neighbors of $\ver$)
    \end{itemize}
    At least $\frac{9}{2} (k-1)$ credits are paid to $\ver$ from horizontal segments $\hor \in \Hor$, such that $\hor.y \in S.y$.

\begin{proof}
    If $S$ contains at least $(k-1)$ generous nodes, which each contribute at least $\frac{9}{2}$ credits ($\cref{alg:close_type}$), the lemma holds. Otherwise, let $S'$ be the set of stingy nodes in $S$. We have that $|S - S'| \leq k-2$.
    
    Let $R \in \dir_{right}(\ver)$ be the set of nodes such that $R.y \subseteq S.y$ and for each $\hor' \in R$, $\hor'.x$ intersects $\bigcap_{\hor \in S'} \hor.x$. According to \cref{lem:gig_box}, $|R| \geq k-1$. Let $\hat{R}$ be the $k-1$ segments $\hor \in R$ with the smallest $\min(\hor.x)$.

    We will see that \cref{alg:estranged_type} pays $\frac{9}{4}$ credits to $\ver$ from every segment $\hor \in \hat{R}$. Let $\hor \in \hat{R}$ be an arbitrary segment.
    Let $\dir_{down}(\hor)\cap S' = S'_{down}$ and $\dir_{up}(\hor)\cap S' = S'_{up}$. Since $|S'| \geq 3(k-1) - (k-2) = 2k - 1$, then either $|S'_{down}| \geq k$ or $|S'_{up}| \geq k$. We consider the case that $|S'_{down}| \geq k$. The case that $|S'_{up}| \geq k$ is symmetric. Let $Q \subseteq \dir_{left}(\hor)$ be the $(k-1)$ rightmost down-heavy vertical segments in \cref{alg:estranged_type}, that are each paid $\frac{9}{4}$ credits from $\hor$. Note that if $\ver \in Q$, then we have shown that \cref{alg:estranged_type} pays $\frac{9}{4}$ credits to $\ver$ from $\hor$. Since, according to our choice of $S$, it is not among the lowest neighbors of $N(\ver)$, then $\ver$ is down-heavy concerning $\hor$. This implies that $\ver\in Q$, unless all nodes of $Q$ are to the right of $\ver$. Since we have reached our goal in the case of $\ver \in Q$, let us now consider only the case that $Q \in \dir_{right}(\ver)$ (\cref{fig:GIG_degree_RD_extended}). Let $\ver' \in Q$ be the vertical segment with the largest $\min(\ver'.y)$. 

    \begin{figure}[h]
        \centering
        \includegraphics[scale=0.18]{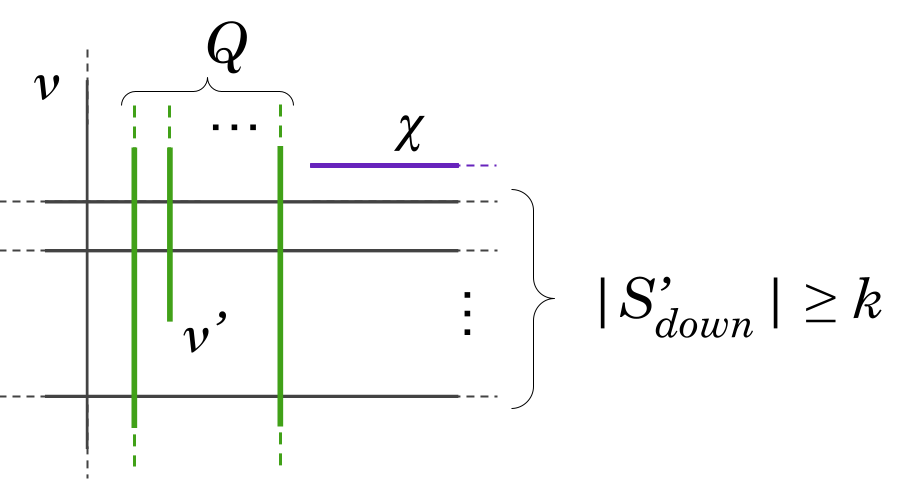}
        \caption{ Depicted are the stingy horizontal segments below $\ver$, the vertical segments $Q$ (green), and $\ver$.}
        \label{fig:GIG_degree_RD_extended}
    \end{figure} 

    Since $\ver$ and some point in $\hor.x$ must intersect $\bigcap_{\hor' \in S'_{down}}\hor'.x$ on the $x$ axis, then all segments in $Q$ (which are between them) also intersect all segments in $S'$ when projected to the $x$-axis. Since $\ver' \in Q$ has the largest $\min(\ver'.y)$, then $S'_{down} \cap N(v')$ and $Q \cup \{\ver\}$ induce a biclique (\cref{fig:GIG_degree_RD_extended}). Since $|Q \cup \{\ver\}| = k$ and there cannot exist a $K_{k, k}$, then $|S'_{down} \cap N(v')| \leq k-1$, which implies that $\min(\ver'.y)$ is greater than $\min(S.y)$ ($\ver'$ does not extend far enough down to include $k$ stingy nodes below $\hor$).
    
    Since $\ver'$ is down-heavy w.r.t $\hor$, then $|N(\ver') \cap \dir_{down}(\hor)| \geq 3(k-1)$ (\cref{def:down_heavy}). Let $\Gamma = N(\ver') \cap \dir_{down}(\hor)$. Note that  $\Gamma.y \subseteq S.y$, because $\min(\ver'.y)$ is greater than $\min(S.y)$ and $\max(\Gamma.y) < \hor.y$. Due to how $S$ was chosen (the last condition), this implies that the nodes in $\Gamma$ that intersect $\ver$ are in $S$. %We have $\Gamma \cap \dir_{right}(\ver) = \Gamma - S$. 
    Since $k$ nodes of $S'_{down} = \dir_{down}(\hor)\cap S'$ cannot be in $N(\ver')$, then $|\Gamma \cap S'| \leq k-1$. Since $|S - S'| \leq k-2$, then $|\Gamma - S| \geq 3(k-1) - (k-1) - (k-2) = k$.

    Since all nodes that intersect $\ver$ in $\Gamma$, must be in $S$, then $\Gamma - S = \Gamma \cap \dir_{right}(\ver) \cup \Gamma \cap \dir_{left}(\ver)$. Since, according to our assumption, $\ver' \in Q \subseteq \dir_{right}(\ver)$, then no horizontal segment in $\Gamma \in N(\ver')$ can be to the left of $\ver$ without intersecting $\ver$. Let $\hat{\Gamma} = (\Gamma - S)\cap \dir_{right}(\ver)$. Then we have $|\hat{\Gamma}| \geq k$. 

    As discussed, when projected to the $x$-axis,  all segments in $Q$ intersect all segments in $S'$.
    We have $\ver' \in Q$ and $\Gamma \in N(\ver')$. Therefore, on the $x$-axis, we have a point $\ver'.x$ that intersects $S'$ as well as $\Gamma$. According to how we defined $R$, we have $\hat{\Gamma} \subseteq R$. 
      
    Since for all $\hor^* \in \Gamma$, $\min(\hor^*.x) \leq v'.x < \min(\hor.x)$, then $\hat{R}$ contains $\hat{\Gamma}$, if it contains $\hor$. This implies that $|\hat{R}| \geq k$, which is a contradiction, since we chose $\hat{R}$ such that it would only contain $k-1$ segments. This implies that the case that $v \not \in Q$ is impossible.

    Let $L \in \dir_{left}(\ver)$ be the set of nodes such that $L.y \subseteq S.y$ and for each $\hor' \in L$, $\hor'.x$ intersects $\bigcap_{\hor \in S'} \hor.x$. According to \cref{lem:gig_box_left}, $|L| \geq k-1$. Let $\hat{L}$ be the $k-1$ segments $\hor \in L$ with the largest $\max(\hor.x)$. The proof that \cref{alg:estranged_type} pays $\frac{9}{4}$ credits to $\ver$ from every segment $\hor \in \hat{L}$ is symmetric.

\end{proof}

\section{Lower bounds}
\label{sec: lower bound}

\begin{lemma}\label{lem:lb_duplication}
    If there exists a $K_{2, 2}$-free $\prig$ graph $G$ on $n$ nodes and $m$ edges, then there exists a $K_{k, k}$-free $\prig$ graph $G'$ on $(k-1)\cdot n$ nodes and $(k-1)^2 \cdot m$ edges.
\end{lemma}
\begin{proof}
Let $G = (U \cup V, E)$ be any $K_{2, 2}$-free graph in $\prig$, and $(\Hor, \Ver, \phi)$ the representation of $G$ in $\prig$, such that $\phi: U \mapsto \Hor, V \mapsto \Ver$,
$\Hor$ is a set of points and $\Ver$ a set of rectangles in $\R^2$. %For every $\hor \in \Hor$, $\ver \in \Ver$, let the projection of $\hor$  to $x$ and $y$ axis be integers $\hor.x$, $\hor.y$ respectively, the projection of $\ver$ to $x$ and $y$ axis be $[\ver_x, \ver'_x]$, $[\ver_y, \ver'_y]$, respectively, where $\ver_x, \ver'_x, \ver_y, \ver'_y$ are integers, such that $\hor.x \neq \ver_x$, $\hor.x \neq \ver'_x$ and similarly for the $y$ axis. Note that this is possible 

We construct a new graph $G = (U' \cup V', E')$, that is represented in $\prig$ by $(\Hor', \Ver', \phi')$, $\phi: U' \mapsto \Hor', V' \mapsto \Ver'$ as follows.
 For every $\hor \in \Hor$, we create $k-1$ copies in $\Hor'$, and  for every $\ver \in \Ver$, we create $k-1$ copies in $\Ver'$. It is easy to see that $|U'| = (k-1) |U|$, $|V'| = (k-1) |V|$. Let us see the number of edges.
 For every $u \in U$, $\phi(u)$ intersects with $|N(u)|$ rectangles in $G$. In $G'$, a copy of $\phi(u)$ intersects with all the copies of $\phi(N(u))$, which number in $(k-1) \cdot |N(u)|$. Let $u'$ be a node in $G'$ that is represented by a copy of $\phi(u)$.  The degree of $u'$ is $k-1$ times larger than the degree of $u$. Since the number of nodes in $U'$ is also $k-1$ times larger than in $U$, then $|E'| = |E| \cdot (k-1)^2$.

 Finally, let us see that $G'$ is $K_{k, k}$-free. Suppose by contradiction that $G'$ has $K_{k, k}$ as a subgraph. Let $S_U \subseteq U'$ and $S_V \subseteq V'$ be the nodes that induce $K_{k, k}$ in $G'$. Then there must be at least two nodes $u, u^* \in U$, whose copies form $S_U$, and at least two nodes $v, v^* \in V$ whose copies form $S_V$. Since $G'[S_U \cup S_V]$ is a biclique, then it has edges between all pairs of nodes. There can be an edge between a pair of nodes $\{u', v'\}$, $u'\in U'$, $v' \in V'$, if and only if that edge existed in $G$ between the nodes whose copies $u'$ and $v'$ are. This implies that $u, u^* \in U$ and $v, v^* \in V$ must have induced a $K_{2, 2}$ in $G$. This is a contradiction.
\end{proof}

Similarly, we can construct copies of other geometric intersection graphs. 
\begin{corollary} \label{cor:lb_duplication_GIG}
    If there exists a $K_{2, 2}$-free $\gig$ graph $G$ on $n$ nodes and $m$ edges, then there exists a $K_{k, k}$-free $\gig$ graph $G'$ on $(k-1)\cdot n$ nodes and $(k-1)^2 \cdot m$ edges.
\end{corollary}

\begin{lemma} \label{lem: chain lower bound}
 $Z_{\chain}(m, n;k) \geq (m+ n)(k-1) - O(k^2)$. 
\end{lemma}
\begin{proof}
 Consider a chain graph $G= (U \cup V, E)$ where $V$ corresponds to the rightward rays $\{r_1, r_2, \ldots, r_{n}\}$ and $U$ to the points $\{p_1, p_2, \ldots, p_{m}\}$. 
Let all the points be placed on the real line such that each point $p_i$ is at coordinate $i$. 
The rays $r_1,\ldots, r_{k-1}$ start at the coordinate $0$, which contains all the points. 
The remaining rays $r_k, \ldots, r_{n}$ start at coordinate $(m-k+1.5)$, so they contain $k-1$ points (see~\Cref{fig:CHAIN_LB_NEW}).
Clearly, the graph does not contain a biclique $K_{k,k}$. The number of edges is $n(k-1) + (m-k+1)(k-1) = (n+m)(k-1)-(k-1)^2$.    
\end{proof}

\begin{figure}[h]
    \centering
    \includegraphics[scale=0.18]{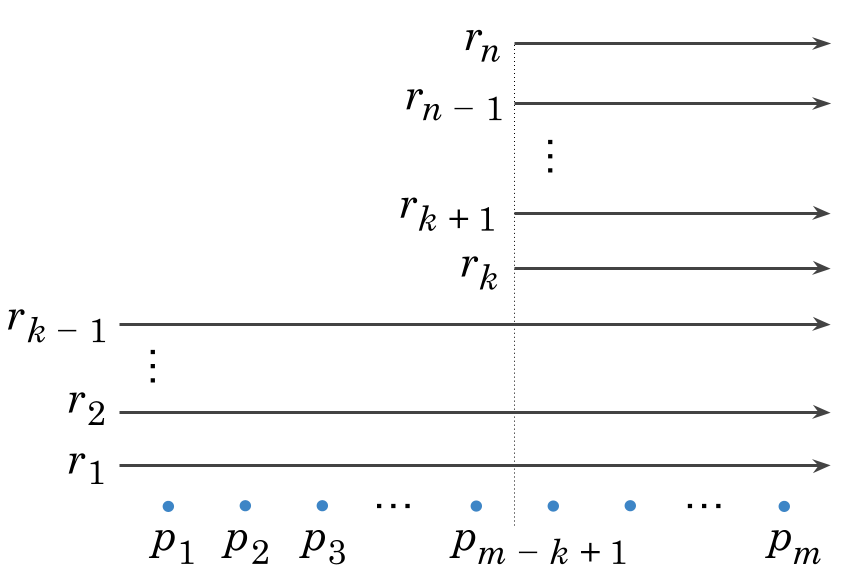}
    \caption{A point-ray representation of the lower bound graph }
    \label{fig:CHAIN_LB_NEW}
\end{figure}

\subsection{Unit Grid Intersection Graphs} \label{sec:gig LB}

Although we do not know the correct bound for the grid intersection graphs, this section presents a lower bound construction that suggests that the leading constant of $Z(n;k)$ is likely at least $3$. This contrasts grid intersection graphs with chordal bigraphs (where the tight bound has a leading constant $2$). 

Our lower bound holds for the special case where all segments have unit length; such an intersection bigraph is called a Unit Grid Intersection Graph ($\ugig$). 

\begin{theorem} 
We have $Z(n;k) \geq (k-1)3n - O(k\sqrt{kn})$ for unit grid intersection graphs.   
\end{theorem}

First, we construct the graph $G'$ for $Z(n;2) \geq 3n - 2\sqrt{n}$, and then we apply $\cref{cor:lb_duplication_GIG}$ to get a graph for $Z(n;k)$. Note that since the method used in the Lemma is duplication, then the new graph created in the proof will also produce a $\ugig$.
For any $t \in \mathbb{N}$ we can construct a $\ugig$ graph as follows (\Cref{fig:GIG_LB}).

\begin{algorithm}[H]
\DontPrintSemicolon
\caption{GIG graph construction}
\label{alg:GIG_LB}
    \textbf{Create} horizontal segments $\Hor$: \ %\;
    \Indp
    \ForEach {$i\in [0, t-1], j \in [0, 2t-1]$}{
        add $[8i, 8i+7]$ at $y$-coordinate $4j+1$ to $\Hor$\;
        add $[8i-4, 8i+3]$ at $y$-coordinate $4j+3$ to $\Hor$\;
    }
    \Indm
    \textbf{Create} vertical segments $\Ver$: \ %\;
    \Indp
    \ForEach {$i, j\in [0, t-1]$}{
        add $[8i+2, 8i+8]$ at $x$-coordinate $8j + 1$ to $\Ver$\;
        add $[8i-2, 8i+4]$ at $x$-coordinate $8j + 2$ to $\Ver$\;
        add $[8i, 8i+6]$ at $x$-coordinate $8j + 5$ to $\Ver$\;
        add $[8i+4, 8i+10]$ at $x$-coordinate $8j + 6$ to $\Ver$\;
    }
    \Indm
\end{algorithm}

\begin{figure}[h]
    \centering
    \includegraphics[scale=0.4]{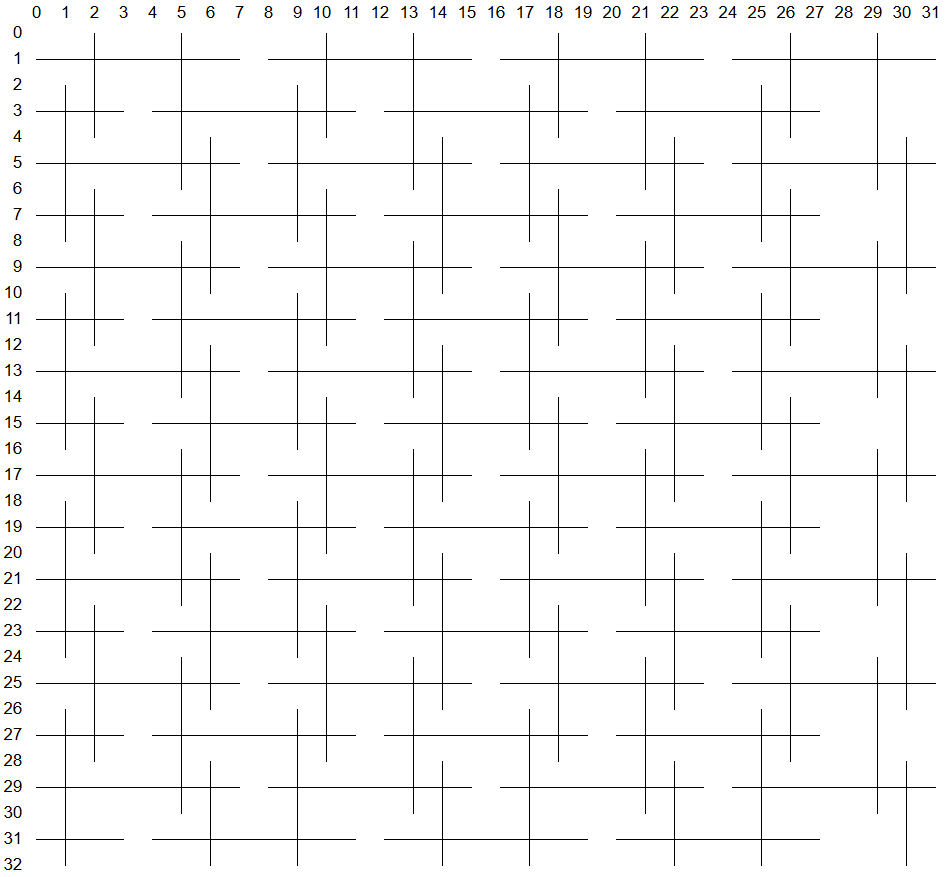}
    \caption{The segments created in \Cref{alg:GIG_LB} for $t = 3$. Only the area of the graph with intersections is depicted. (Some segments also have negative coordinates.) }
    \label{fig:GIG_LB}
\end{figure}

\begin{figure}[h]
    \centering
    \includegraphics[scale=0.16]{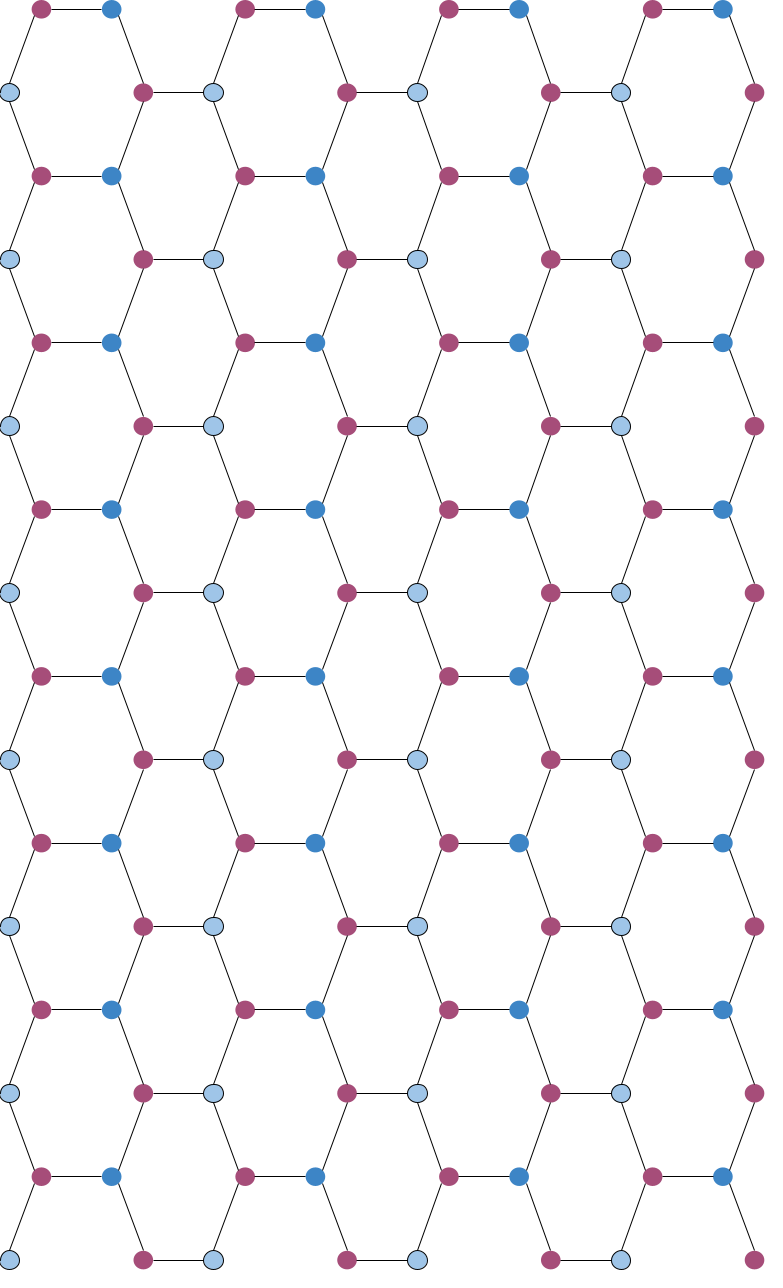}
    \caption{A planar representation of the graph $G'$, where the first and second types of horizontal segments are depicted by blue and light blue nodes, respectively, and red nodes depict vertical segments. }
    \label{fig:GIG_LB_planar}
\end{figure}

Note that $|\Hor| = t\cdot 2t + t\cdot 2t = 4t^2$ and $|\Ver| = t^2 + t^2 + t^2 + t^2 = 4t^2$. To count the number of edges, we count the edges incident to horizontal segments $\Sum_{\hor \in \Hor} \deg_{G'}(\hor)$. Consider the segments that are added by the first line in the algorithm. For these segments, the degree for $j = 0$ (y-coordinate $1$) is $2$, and for all others, the degree is $3$ (\Cref{fig:GIG_LB}). The number of edges incident to the first type of horizontal segments is $3\cdot 2t^2-t$. Now, consider the segments added by the second line in the algorithm. For these segments, all cases where $i = 0$ have one degree less than $3$, as well as the cases where $j = 2t-1$ (the case where both are true has degree $3-2 = 1$, see \Cref{fig:GIG_LB}). In total,  the number of edges incident to these segments is therefore $3\cdot 2t^2 - 2t -t$.
The constructed graph has $|\Hor| = |\Ver| = 4t^2$ and the number of edges is $\Sum_{\hor \in \Hor} \deg_{G'}(\hor) = 12t^2 - 4t$.

\begin{figure}[h]
    \centering
    \includegraphics[scale=0.2]{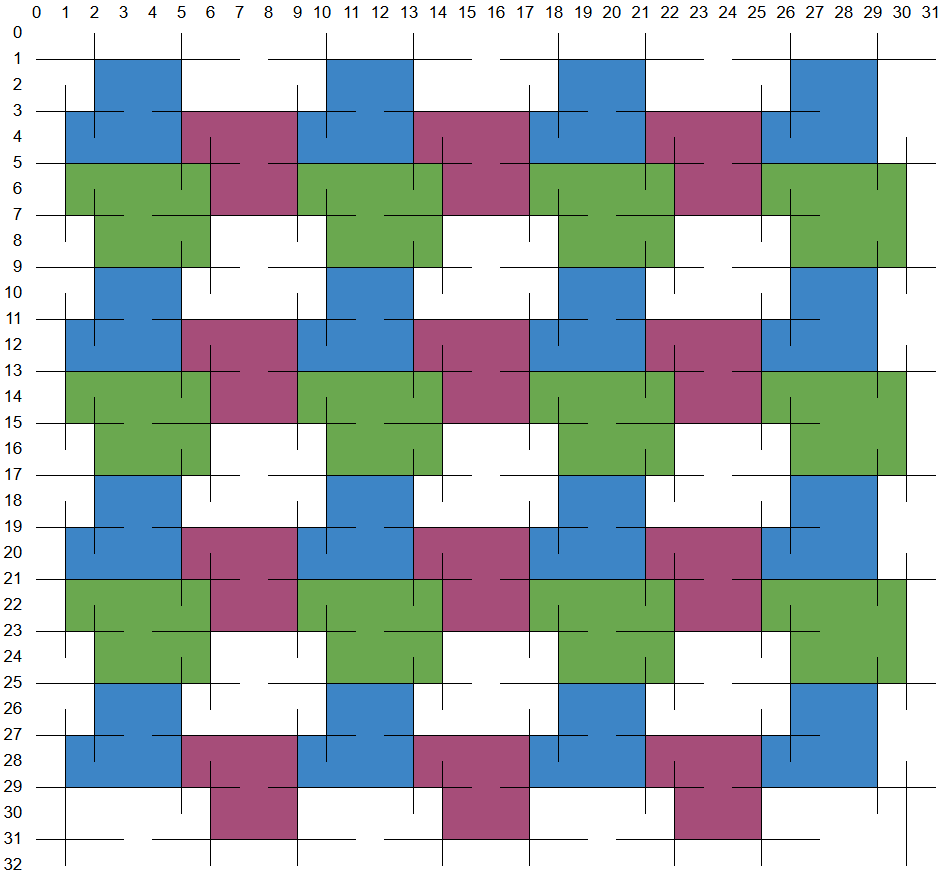}
    \caption{Areas bordered by segments are colored white, blue, green, or red. }
    \label{fig:GIG_LB_colored}
\end{figure}

Let \textbf{faces} be the areas bordered by segments.
To see that the graph is $K_{2, 2}$-free, note that a $\gig$ representation of a $K_{2, 2}$ would form a rectangle. Since segments in each partition have identical lengths, a rectangle that contains only non-rectangle faces cannot exist. Since none of the faces in the graph are rectangles (\Cref{fig:GIG_LB_colored}), the graph is $K_{2, 2}$-free. This finishes the proof for $Z(n;2) \geq 3n - 2\sqrt n$.

Next, we use \cref{cor:lb_duplication_GIG}, to create a $K_{k,k}$-free graph $G = (U \cup V, E)$, where $|U| = |V| = 4(k-1)t^2$ and $|E| = (k-1)^2 (12t^2 - 4t)$.
%Next, we create a graph $G = (U \cup V, E)$ where each segment of the previous graph is duplicated $k-1$ times. Let $\sigma: U\cup V \mapsto \Hor \cup \Ver$ be a function that signifies which original segments produced the node in $U \cup V$. Clearly, $|U| = |V| = (k-1)4t^2$. Since the degree of each node in $G$ is $(k-1)$ times larger than in $G'$, then $|E| = \Sum_{u \in U} \deg_G(u) = \Sum_{u \in U} (k-1)\deg_{G'}(\sigma(u))$. Since there are $k-1$ elements $u \in U$ mapping to the same $\sigma(u)$, then $|E| =  (k-1)^2\Sum_{x \in \Hor} \deg_{G'}(x) = (k-1)^2 (12t^2 - 4t)$. 
Using the value $n = 4(k-1)t^2$, we get $|E| >  (k-1)(3n - 2\sqrt{(k-1)n} )$. 
%It remains to see that the graph $G$ is $K_{k,k}$-free. By contradiction, suppose it includes a $K_{k,k}$. Since each segment is duplicated $k-1$ times, there must be at least $2$ segments in each, $\Hor$ and $\Ver$, whose copies form the $K_{k,k}$. Since these $4$ segments have to form a biclique in the original graph to form a biclique in $G$, this is a contradiction, because the original graph is $K_{2, 2}$-free.

\end{document}